\documentclass[12pt,a4paper]{article}
\usepackage[english]{babel}
\usepackage{amsmath}
\usepackage{amsfonts}
\usepackage{amssymb}
\usepackage{amsthm}
\usepackage{mathrsfs}
\usepackage{graphicx}
\usepackage{caption}
\usepackage[left=2cm,right=2cm,top=2cm,bottom=2cm]{geometry}
\usepackage{indentfirst}

\newcommand{\brho}{{\boldsymbol \rho}}

\newcommand{\bbA}{{\bold A}}

\newcommand{\bbD}{{\bold D}}

\newcommand{\bbS}{{\bold S}}
\newcommand{\bbR}{{\bold R}}
\newcommand{\bbQ}{{\bold Q}}

\newcommand{\bV}{{\boldsymbol V}}

\newcommand{\bc}{{\boldsymbol c}}
\newcommand{\bq}{{\boldsymbol q}}
\newcommand{\br}{{\boldsymbol r}}
\newcommand{\bu}{{\boldsymbol u}}
\newcommand{\bv}{{\boldsymbol v}}
\newcommand{\bw}{{\boldsymbol w}}

\newcommand{\bz}{{\boldsymbol z}}

\newcommand{\diag}{\mathrm{diag}\,}
\newcommand{\const}{\mathrm{const}\,}
\newcommand{\sign}{\mathrm{sign}\,}

\newcommand{\ol}[1]{\overline{#1}}

\newtheorem{Supp}{Proposition}[section]
\newtheorem{Note}{Remark}[section]

\newtheorem{Theorem}{Theorem}[section]

\title{Control of body motion in an ideal fluid using the internal mass and the rotor in the presence of circulation around the body}

\author{E.V. Vetchanin$^1$, A.A. Kilin$^2$}
\date{}

\begin{document}
\maketitle

\vspace{-10mm}

\begin{center}
$^1$ Kalashnikov Izhevsk State Technical University

$^2$ Udmurt State University
\end{center}

\begin{small}
\textbf{Abstract.} In this paper we study the controlled motion of an arbitrary two-dimensional body in an ideal fluid with a moving internal mass and an internal rotor in the presence of constant circulation around the body.  We show that by changing the position of the internal mass and by rotating the rotor, the body can be made to move to a given point, and discuss the influence of nonzero circulation on the motion control. We have found that in the presence of circulation around the body the system cannot be completely stabilized at an arbitrary point of space, but fairly simple controls can be constructed to ensure that the body moves near the given point.
\end{small}

\tableofcontents

\section{Introduction}

The problem of the motion of a rigid body in an ideal fluid is a classical problem of hydrodynamics and has been studied for a long time. Many significant results were obtained within the model of an ideal fluid by Kirchhoff~\cite{Kirchhoff}, Lamb~\cite{Lamb}, Chaplygin~\cite{Chaplygin}, and Steklov~\cite{Steklov}. For a modern qualitative analysis of the motion of a rigid body in an ideal fluid, see~\cite{Bor_Mam_2006, BorKozMam_2007, cite_23}. From a practical point of view, it is interesting to study the problem of controllable motion of a rigid body in a fluid. Viscosity is of significant importance for producing traction force \cite{Childress,VMT_2013}, but some interesting control methods can be explained using ideal fluid theory. For instance, a theoretical proof of the propulsion of a body by moving the internal masses is given in \cite{Kozlov_2001}. The results of \cite{Kozlov_2001} are extended in \cite{Kilin_Ram_Ten, Kozlov_2003}. Another method of moving the body involves changing the gyrostatic momentum using the rotation of internal rotors. The application of rotors for the stabilization of motion of an underwater vehicle is examined in \cite{Woolsey_Leonard_2002}. Issues of the stability of equilibria of a neutral buoyancy underwater vehicle are considered in \cite{Leonard_Marsden, Leonard}. Results of advanced investigations of the control of nonholonomic systems by means of internal rotors are presented in \cite{Bolotin, BKM_Chap_2012, BKM_Chap_2013, Ivanov_2013, Svinin_2012}.

When the body's motion in an ideal fluid is examined, nonzero circulation of the velocity of a fluid around the body is sometimes assumed. Such a problem statement goes back to Zhukovskii and Chaplygin and explains the lift force acting on the wing. The nonzero circulation results in the appearance of gyroscopic forces acting on the body \cite{Bor_Mam_2006}, which change the dynamics of the system significantly. The controlled motion using an internal rotor changing the proper gyrostatic momentum of the system and using a Flettner rotor changing the circulation around the body is considered in \cite{Ram_Ten_Tre}. In \cite{Vetchanin_Kilin}, analysis of the free motion of a hydrodynamically asymmetric body in an ideal fluid in the presence of circulation is performed and the controllability of the motion by changing the position of the center of mass of the system is proved.

An essential part of the investigation of the controlled motion is the proof of the possibility of control. In \cite{BKM_Chap_2012, Kozlov_2001, Murray_Sastry_1993, Vetchanin_Kilin} the Rashevskii--Chow theorem is used to prove controllability \cite{Agrachev_Sachkov_2004}. In the original form this theorem applies only to driftless systems. In cases where the system is subject to drift, a modification of the Rashevskii--Chow theorem is used for the proof of controllability \cite{Bonnard}. In addition to the requirements of the initial theorem, the requirements of the modification involve proving the Poisson stability of drift. This modification was applied, for example, in \cite{Crouch}.

The investigation of controlled motion leads to the operation speed problem. For example, it is of interest to construct time-optimal controls \cite{Leonard_Optimal, Chyba_Leonard_Sontag}. For various aspects of control theory, see \cite{Agrachev_Sachkov_2004, Jurdjevic}.

This paper is an extension of \cite{Vetchanin_Kilin}, but in contrast to the previous paper, we are concerned with a  combined control scheme which uses the motion of the internal mass and rotation of the internal rotor. In Section~2, equations of motion and their first integrals are presented. In Section~3 the controllability is proved. It is shown that although controllability can be achieved by using the rotor alone, it is not constructive since it depends considerably on the drift. Note that in this case nonzero circulation facilitates controllability. Therefore, we have proved the controllability by means of the internal mass and the internal rotor in the following cases: the internal mass moves in a circle (cam), and the internal mass moves in a straight line (slider). These motion patterns have been selected because of ease of technical implementation. In Section~4, we show a negative influence of circulation, which leads to the necessity of applying controls to stabilize the body at the end point of the trajectory. In the previous paper \cite{Vetchanin_Kilin} we have shown that the internal mass needs to be moved uniformly in a straight line to ensure that the body is stabilized at a given point. Since the motion of the internal mass is bounded by the body's boundary, this method of stabilization is not reasonable. In view of the above, it is proposed to ensure, by using controls, zero translational velocity of the body with its angular velocity being nonzero. The necessary calculations are presented in Appendices A and B.

\section{The mathematical model}

Consider the two-dimensional problem of motion in an infinite volume of an ideal incompressible fluid of a hydrodynamically asymmetric body with mass $M$ and central moment of inertia $I$ (see Fig. \ref{frames}).

\begin{figure}[h!]
	\begin{center}
		\includegraphics[width=0.5\linewidth]{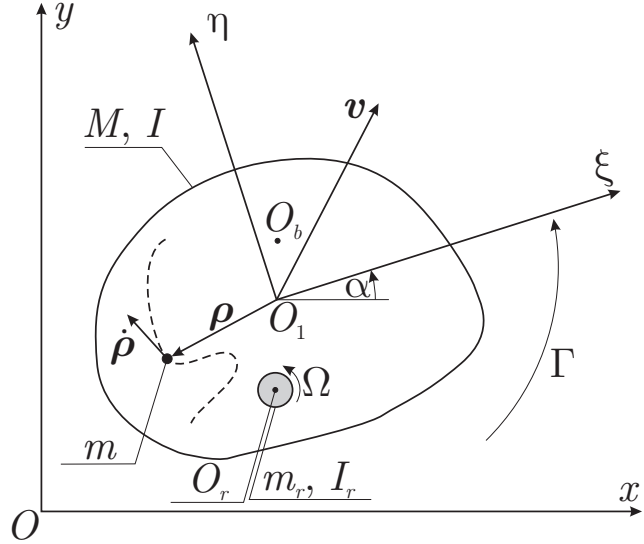}
		\caption{Body with an internal mass and an internal rotor}\label{frames}
	\end{center}
	\vspace{-5mm}
\end{figure}

The body carries a particle of mass $m$ and a rotor with mass $m_r$ and central moment of inertia $I_r$. The motion of the particle is limited by the shell, but the particle follows an arbitrary smooth trajectory $\brho=(\xi(t),\, \eta(t))$. The rotor has the shape of a circular cylinder, is homogeneous, rotates with angular velocity $\Omega(t)$; its axis of rotation is perpendicular to the plane of motion of the body and passes through the center of mass of the rotor. We assume that there is a nonzero and constant (by the Lagrange theorem) circulation $\Gamma$ of the fluid velocity around the body.

%


To describe the motion of the system, we introduce two Cartesian coordinate systems: a fixed one, $Oxy$, and a moving one, $O_1\xi\eta$, attached to the body (see Fig. \ref{frames}). Point $O_1$ coincides with the position of the center of mass of the body--rotor system. The position of the body in absolute space is characterized by the radius vector $\br = (x,\, y)$ and the angle of rotation $\alpha$ of the moving coordinate system relative to the fixed coordinate system. Thus, the configuration space of the system is $\mathcal{H} = \mathbb{R}^2 \times \mathbb{S}^1$,
and the pair $(\br,\, \alpha)$ completely specifies the position and orientation of the body.

Let $\bv = (v_1,\, v_2)$ denote the absolute velocity of point $O_1$ of the body referred to the axes of the moving coordinate system and $\omega$ the angular velocity of the body. Then the following kinematic relations hold:
\begin{gather}
\dot{\bq} = \bbQ \bw, \quad \bbQ = \begin{pmatrix}
\cos \alpha & -\sin \alpha & 0\\
\sin \alpha & \cos \alpha & 0\\
0 & 0 & 1
\end{pmatrix}\label{kinem}
\end{gather}
where $\bq = (x,\, y,\, \alpha)$ is the vector of generalized coordinates and $\bw = (v_1,\, v_2,\, \omega)$ is the vector of quasi-velocities.

We specify the position of the body's center of mass (point $O_b$) by the radius vector $(\xi_b,\, \eta_b)$ and the position of the rotor's center of mass (point $O_r$) by the radius vector $(\xi_r,\, \eta_r)$. In the chosen coordinate system the coordinates of points $O_b$ and $O_r$ are related by
\begin{gather}
M \xi_b + m_r \xi_r =0,\quad M \eta_b + m_r \eta_r =0.\nonumber
\end{gather}
The expressions for the kinetic energy of the body $T_b$, the internal mass $T_m$, the internal rotor $T_r$ and the fluid $T_f$ have the form
\begin{gather}
T_b = \frac{1}{2} M \left( (v_1-\eta_b \omega)^2 + (v_2+\xi_b\omega)^2\right) + \frac{1}{2} I \omega^2, \nonumber\\
T_m = \frac{1}{2} m \left( (v_1 +\dot{\xi} - \eta \omega)^2 + (v_2 + \dot{\eta} + \xi \omega)^2 \right), \nonumber\\
T_r = \frac{1}{2} m_r \left( (v_1 - \eta_r \omega)^2 + (v_2 + \xi_r\omega)^2 \right) + \frac{1}{2} I_r (\omega + \Omega)^2 , \nonumber\\
T_f = \frac{1}{2}\lambda_1 v_1^2 + \frac{1}{2} \lambda_2 v_2^2 + \frac{1}{2} \lambda_6 \omega^2, \nonumber
\end{gather}
where $\lambda_1$ and $\lambda_2$ are the added masses ($\lambda_1 \neq \lambda_2$), $\lambda_6$ is the added moment of inertia, and $\xi(t)$, $\eta(t)$, $\Omega(t)$ are known functions of time which play the role of controls in the system considered. With our choice of the origin of the moving coordinate system, the kinetic energy  of the entire system is defined, up to the known functions of time, by the following expression:
\begin{gather}
T = \frac{1}{2} \left(\bbA\bw,\, \bw \right) + \left(\bu,\, \bw\right),\\
\bbA = \begin{pmatrix}
a_1 & 0 & f\\
0 & a_2 & g\\
f & g & b
\end{pmatrix}, \quad \bu = \begin{pmatrix}
m \dot{\xi}\\
m \dot{\eta}\\
m (\xi \dot{\eta} - \dot{\xi} \eta) + I_r \Omega
\end{pmatrix}, \nonumber\\
a_1 = M + m + m_r + \lambda_1,\quad a_2 = M + m + m_r + \lambda_2, \nonumber \\
b = M (\xi_b^2 + \eta_b^2) + I + m(\xi^2 + \eta^2) + m_r (\xi_r^2 + \eta_r^2) + I_r + \lambda_6, \quad f = -m \eta,\quad g = m \xi.\nonumber
\end{gather}

The equations of motion of the system incorporating forces due to circulation around the body have the form
\begin{gather}
\begin{gathered}
\frac{d}{dt}\left( \frac{\partial T}{\partial v_1} \right) = \omega \frac{\partial T}{\partial v_2} - \lambda v_2 -\zeta \omega,\\
\frac{d}{dt}\left( \frac{\partial T}{\partial v_2} \right) =- \omega \frac{\partial T}{\partial v_1} + \lambda v_1 + \chi \omega,\\
\frac{d}{dt}\left( \frac{\partial T}{\partial \omega} \right) = v_2 \frac{\partial T}{\partial v_1} - v_1 \frac{\partial T}{\partial v_2} +\zeta v_1 - \chi v_2,
\end{gathered}\label{T_eq}
\end{gather}
where $\lambda = \rho \Gamma$, $\zeta = \rho \Gamma \mu$, $\chi = \rho \Gamma \nu$, and $\mu$, $\nu$ are the coefficients associated with the hydrodynamical asymmetry of the body \cite{Chaplygin}.

Equations \eqref{T_eq} can be written in the form of Poincar\'{e} equations on the group $E(2)$
\begin{gather}
\begin{gathered}
\frac{d}{dt}\left( \frac{\partial L}{\partial v_1} \right) = \omega \frac{\partial L}{\partial v_2} + \cos\alpha \frac{\partial L}{\partial x} + \sin\alpha \frac{\partial L}{\partial y},\\
\frac{d}{dt}\left( \frac{\partial L}{\partial v_2} \right) = - \omega \frac{\partial L}{\partial v_1} - \sin\alpha \frac{\partial L}{\partial x} + \cos\alpha \frac{\partial L}{\partial y},\\
\frac{d}{dt}\left( \frac{\partial L}{\partial \omega} \right) = v_2 \frac{\partial L}{\partial v_1} - v_1 \frac{\partial L}{\partial v_2} + \frac{\partial L}{\partial \alpha}
\end{gathered}\label{L_eq}
\end{gather}
with Lagrangian
\begin{gather}
L = \frac{1}{2} \left(\bbA\bw,\, \bw \right) + \left( \bc,\, \bw \right) + \left(\bu,\, \bw\right),\\
\bc = \begin{pmatrix}
- \frac{\lambda}{2}(x\sin\alpha -y \cos \alpha)\\
-\frac{\lambda}{2} (x\cos \alpha + y\sin \alpha)\\
-\chi (x\sin\alpha -y\cos\alpha)-\zeta(x\cos\alpha+y\sin\alpha)
\end{pmatrix}.\nonumber
\end{gather}

Equations \eqref{kinem}, \eqref{L_eq} form a closed system of six equations
\begin{gather}
\begin{gathered}
\frac{d}{d t} (a_1 v_1 + f \omega + m \dot{\xi}) = \omega (a_2 v_2 + g \omega + m \dot{\eta}) - \lambda v_2 - \zeta \omega, \\
\frac{d}{d t} (a_2 v_2 + g \omega + m \dot{\eta}) =  -\omega (a_1 v_1 + f \omega + m \dot{\xi}) + \lambda v_1 + \chi \omega,\\
\begin{split}
\frac{d}{dt} (f v_1 + g v_2 + b \omega + m (\xi \dot{\eta} - \eta \dot{\xi}) + I_r \Omega) = {} & {} \\
= v_2 (a_1 v_1 + f \omega + m \dot{\xi}) - {} & {} v_1 (a_2 v_2 + g \omega + m \dot{\eta}) + \zeta v_1 - \chi v_2,
\end{split}\\
\frac{d x}{dt} =  v_1 \cos \alpha - v_2 \sin \alpha,\quad \frac{d y}{dt} =  v_1 \sin \alpha + v_2 \cos\alpha,\quad \frac{d \alpha}{dt} = \omega
\end{gathered}\label{exp_full}
\end{gather}
in the variables $v_1$, $v_2$, $\omega$, $x$, $y$, $\alpha$ and completely describe the motion of the system considered.

For the case of a freely moving system, i.e., for $\dot{\xi} = \dot{\eta} = 0$, $\Omega = 0$, equations \eqref{exp_full} admit the first integrals \cite{Chaplygin, Bor_Mam_2006}
\begin{gather}
\begin{gathered}
p_x = \left( \frac{\partial L}{\partial v_1} - \chi \right) \cos\alpha - \left( \frac{\partial L}{\partial v_2} - \zeta \right) \sin\alpha + \frac{\lambda}{2}y,\\
p_y = \left( \frac{\partial L}{\partial v_1} - \chi \right) \sin\alpha + \left( \frac{\partial L}{\partial v_2} - \zeta \right) \cos\alpha - \frac{\lambda}{2}x,\\
K = x p_y - y p_x + \frac{\partial L}{\partial \omega} + \frac{\lambda }{2} (x^2 + y^2) - c_3.
\end{gathered}\label{imp_int}
\end{gather}
These integrals are generalized to the case of controlled motion and can be written in explicit form as follows:
\begin{gather}
\begin{gathered}
p_x = (a_1 v_1 + f\omega + u_1 - \chi) \cos\alpha - (a_2 v_2 + g\omega + u_2 - \zeta)\sin\alpha + \lambda y,\\
p_y = (a_1 v_1 + f\omega + u_1 - \chi) \sin\alpha + (a_2 v_2 + g\omega + u_2 - \zeta)\cos\alpha - \lambda x,\\
K = f v_1 + g v_2 + b \omega + u_3 + \frac{\lambda}{2}(x^2+y^2)+x p_y - y p_x.
\end{gathered}\label{exp_int}
\end{gather}
The integrals $p_x$, $p_y$, $K$ have the meaning of the linear and angular momentum components of the body + control elements + fluid system and are a generalization of the integrals for a system with moving internal masses \cite{Kozlov_2001}.

Note that \eqref{exp_full} contains only the derivative $\dot\Omega$ (and not the angular velocity $\Omega$ itself), hence, the rotor rotating with constant angular velocity does not influence the dynamics of the system.

+\section{Controllability}

To prove the controllability of motion on the fixed level set of first integrals, we shall use a modification of the Rashevskii-Chow theorem \cite{Agrachev_Sachkov_2004,Chow_1939,Rashevskii} for systems with drift\footnote{The term {\it drift} is used to mean nonzero motion of the system with control disabled.} \cite{Bonnard}. This theorem requires, in addition to completeness of the linear span of the vector fields and their commutators, that there exists everywhere a dense set of Poisson stable points for the free motion (drift) in the phase space of the system.

The issue of Poisson stability of drift is considered in our previous paper \cite{Vetchanin_Kilin}. In particular, it was shown that on the common level set of the integrals $p_x$ and $p_y$ the velocities $v_1$, $v_2$, and $\omega$ are related to the coordinates $x$ and $y$ by
\begin{gather}
(a_1 v_1 + f \omega - \chi)^2 + (a_2 v_2 + g \omega - \zeta)^2 = (p_x - \lambda y)^2 + (p_y + \lambda x)^2 \label{eq.compactness}
\end{gather}
It is clear from \eqref{eq.compactness} that the free motion is bounded by a circular region whose size and position depend on the level of the kinetic energy of the system $T$, the level sets of the integrals $p_x$ and $p_y$, the body geometry and circulation. In addition, the system of equations \eqref{kinem} and \eqref{L_eq} is integrable; hence, by the Poincar\'{e} recurrence theorem \cite{Arnold}, the free motion of the system is Poisson stable.
Therefore, in what follows we shall investigate only the issue of completeness of the linear span of the vector fields and their commutators.

In \cite{Vetchanin_Kilin} the controllability of motion by means of an internal mass capable of moving arbitrarily inside the body is proved. Therefore, an analogous system to which an internal rotor is added is controllable as well. In this section we prove controllability for two particular cases in which the following restrictions are imposed on possible motions of the internal mass:
\begin{enumerate}
	\itemsep=-2pt
	\item The internal mass is fixed.
	\item The internal mass moves along a given curve.
\end{enumerate}

\subsection{The case of a fixed internal mass} \label{subsec_rot}

Let us examine the controllability of the system's motion only by changing the rotation of the rotor. In this case, $\dot{\xi} = \dot{\eta} = 0$, and the equations of motion \eqref{L_eq} and the first integrals \eqref{exp_int} are
\begin{gather}
\begin{gathered}
\frac{d}{dt}\left( a_1 v_1 + f\omega \right) = \omega (a_2 v_2 + g \omega) - \lambda v_2 -\zeta \omega,\\
\frac{d}{dt}\left( a_2 v_2 + g\omega \right) = - \omega (a_1 v_1 + f \omega) + \lambda v_1 + \chi \omega,\\
\frac{d}{dt}\left( f v_1 + g v_2 + b \omega + I_r \Omega \right) = v_2 (a_1 v_1 + f \omega) - v_1 (a_2 v_2 + g \omega) +\zeta v_1 - \chi v_2
\end{gathered}
\end{gather}
and
\begin{gather}
\begin{gathered}
p_x = (a_1 v_1 + f\omega - \chi) \cos\alpha - (a_2 v_2 + g\omega - \zeta)\sin\alpha + \lambda y,\\
p_y = (a_1 v_1 + f\omega - \chi) \sin\alpha + (a_2 v_2 + g\omega - \zeta)\cos\alpha - \lambda x,\\
K = f v_1 + g v_2 + b \omega + I_r \Omega + \frac{\lambda}{2}(x^2+y^2)+x p_y - y p_x.
\end{gathered}\label{exp_int1}
\end{gather}

From the integrals $p_x$, $p_y$, and $K$ we express the velocities
\begin{gather}
\bw = \bbA^{-1} \begin{pmatrix}
\ol{x} \cos\alpha + \ol{y}\sin\alpha + \chi\\
-\ol{x} \sin\alpha + \ol{y}\cos\alpha + \zeta\\
F - \frac{1}{2\lambda}(\ol{x}^2 + \ol{y}^2) - I_r \Omega
\end{pmatrix},\label{vel_rot}
\end{gather}
where $\ol{x} = p_x - \lambda y$, $\ol{y} = p_y + \lambda x$, $F = \dfrac{2 K \lambda + p_x^2 + p_y^2}{2\lambda}$. Substituting \eqref{vel_rot} into the kinematic relations \eqref{kinem}, we obtain the equations of motion for the body on the fixed level set of the integrals $p_x$, $p_y$, $K$ in a standard form linear in the controls
\begin{gather}
\dot{\bq} = \bV_0(\bq) + \bV_1 \Omega, \label{rot_cont}\\
\bV_0(\bq) = \bbS \left( \ol{x} \cos \alpha + \ol{y} \sin\alpha + \chi,\;  -\ol{x} \sin\alpha + \ol{y} \cos\alpha + \zeta,\; F - \frac{1}{2\lambda} (\ol{x}^2 + \ol{y}^2) \right)^T , \nonumber\\
\bV_1 = \bbS (0,\; 0,\; -I_r)^T,  \quad \bbS = \bbQ \bbA^{-1}.\nonumber
\end{gather}
Here the angular velocity of rotation of the rotor $\Omega$ is considered as control, the vector field $\bV _0$ corresponds to the free motion (drift), and the vector field $\bV _1$ is related to the control action. Consider the vector fields
\begin{gather}
\begin{gathered}
\bV_0,\quad \bV _1, \quad \bV_{2}=\left[ \bV_0,\, {\bV}_1 \right], \quad \bV_{3} =\left[\bV_2,\, {\bV}_1 \right],
\end{gathered} \label{rot_vect}
\end{gather}
where $\left[ \cdot,\, \cdot \right]$ is the Lie bracket. The rank of the linear span of the vector fields $\bV_0,\, \bV_1,\, \bV_3$ is equal to three everywhere except on the surface given by
\begin{gather}
\begin{gathered}
2 (a_2^2 - a_1^2) \left( (\ol{x}^2 - \ol{y}^2) \sin\alpha \cos\alpha - \ol{x}\ol{y}(\cos^2\alpha-\sin^2\alpha) \right) + \\
+ a_1 (2a_1 - a_2) \zeta (\ol{y}\sin\alpha + \ol{x}\cos\alpha) + a_2 (2 a_2 - a_1) \chi (\ol{x} \sin\alpha - \ol{y} \cos\alpha) = 0.
\end{gathered}\label{surf1}
\end{gather}
In a similar way, the rank of the linear span of the vector fields $\bV_1,\, \bV_2,\, \bV_3$ is equal to three everywhere except on the surface given by
\begin{gather}
\begin{gathered}
a_1^2 \zeta^2 + a_2^2 \chi^2 + 2 (a_1 - a_2)^2 (\ol{x}^2 + \ol{y}^2) + \\
+ 3a_1 (a_2 - a_1) \zeta (\ol{x}\sin\alpha - \ol{y}\cos\alpha) + 3 a_2 (a_2 - a_1) \chi (\ol{y} \sin\alpha + \ol{x}\cos\alpha)=0.
\end{gathered}\label{surf2}
\end{gather}

\begin{Note}
	The expressions \eqref{surf1} and \eqref{surf2} hold in the case where the moving coordinate system is chosen such that the matrix $\bbA$ takes the diagonal form $\bbA = \diag (a_1,\, a_2,\, b)$.
\end{Note}

It is easy to show that the surfaces \eqref{surf1} and \eqref{surf2} intersect along the curves
\begin{gather}
\begin{gathered}
\ol{x} = \left(\tau_1 + \tau_2 \right)\sin \alpha + \left(\tau_1 - \tau_2 \right)\cos \alpha,\\
\ol{y} = - \left(\tau_1 + \tau_2 \right)\cos \alpha + \left(\tau_1 - \tau_2 \right)\sin \alpha,
\end{gathered}\label{low_curve}
\end{gather}
where $\tau_1$ and $\,\tau_2$ are the solution of the system of equations
\begin{gather}
\begin{gathered}
\left( 8 \vert a_2 - a_1 \vert \tau_1 + 3 (a_1 \zeta + a_2 \chi) \right)^2 + \left( 8 \vert a_2 - a_1 \vert \tau_2 + 3 (a_1 \zeta - a_2 \chi) \right)^2 - 2 \left( a_1^2 \zeta^2 + a_2^2 \chi^2 \right) = 0,\\
\left(\tau_1 + \frac{a_2 (2a_2 - a_1) \chi + a_1 (2a_1 - a_2) \zeta}{8 (a_2^2 - a_1^2)} \right)^2- \left(\tau_2 - \frac{a_2 (2a_2 - a_1) \chi - a_1 (2a_1 - a_2) \zeta}{8 (a_2^2 - a_1^2)} \right)^2 - \\
\\ - \frac{4a_1 a_2 (2a_2 - a_1) (2a_1 - a_2) \chi \zeta}{16 (a_2^2 - a_1^2)} = 0.
\end{gathered}\label{rot_sing}
\end{gather}
Thus, in the configuration space $\mathcal{H}$ the dimension of the linear span of the vector fields \eqref{rot_vect} is equal to three everywhere except along the curves \eqref{low_curve}. Since the above curves are the surfaces of codimension two, the following theorem holds.
\begin{Theorem}\label{theo.rotor}
	An arbitrary body moving in a fluid (in the presence of circulation around the body) with a given initial velocity can be moved by an appropriate rotation of the internal rotor from any initial position to any end position.
\end{Theorem}

We note that the controllability proved in Theorem \ref{theo.rotor} is formal. Indeed, to construct control, the Rashevskii--Chow theorem uses the motion along the vector fields in both forward and backward time. When there is a drift, the motion along it is possible only in forward time. The motion in backward time is implemented by using the recurrence property of the trajectories. However, the calculation of the recurrence period in specific systems can be a fairly complicated problem (which can even defy a solution, for example, in the case of chaotic systems with measure), and the period itself can be very large or even tend to infinity. Thus, an appropriate construction of such controls is impossible.

For the chosen control method the presence of circulation is a necessary condition for controllability. Indeed, it is easy to show that for $\Gamma = 0$ the system \eqref{rot_cont} becomes uncontrollable in the sense of Rashevskii--Chow. On the other hand, it follows from \eqref{rot_cont} that the drift caused by circulation cannot be completely compensated for by rotating the rotor. Therefore, in what follows we consider a combined model of controlling by both the rotor and the moving internal mass.

\subsection{The case of motion of the internal mass along a given curve}

We shall assume that the internal mass can move only along some curve $\brho = (\xi(s), \eta(s))$, where $s$ is a parameter of the curve. From the integrals \eqref{exp_int} we express the velocities
\begin{gather}
\bw = \bbA^{-1} \begin{pmatrix}
\overline{x} \cos \alpha + \overline{y} \sin \alpha + \chi - m \frac{d \xi}{d s} \dot{s}\\
-\overline{x} \sin \alpha + \overline{y} \cos\alpha + \zeta - m \frac{d \eta}{d s} \dot{s}\\
F - \frac{\overline{x}^2+\overline{y}^2}{2\lambda} - m (\xi \frac{d \eta}{d s} - \eta \frac{d \xi}{ds}) \dot{s} - I_r \Omega
\end{pmatrix}. \label{veloc}
\end{gather}
Substituting \eqref{veloc} into the kinematic relations \eqref{kinem}, we obtain the equations of motion for the body on the fixed level set of the integrals \eqref{exp_int}. These equations depend on both the position $s$ of the mass on the curve and its velocity $\dot{s}$. Applying the standard method for phase space extension (Goh transformation \cite{Bonnard_Chyba}), we obtain the equations of motion in the form linear in the controls
\begin{gather}
\dot{\bz} = \bV_0(\bz)+\bV_1(s) \dot{s} + \bV_2 \Omega, \label{kinem_1}\\
\bV_0(\bz) = \bbR \left( \overline{x} \cos \alpha + \overline{y} \sin \alpha + \chi,\, -\overline{x} \sin \alpha + \overline{y} \cos\alpha + \zeta,\, F - \frac{\overline{x}^2+\overline{y}^2}{2\lambda},\, 0 \right)^T,\nonumber\\
\bV_1(s) = \bbR \left( -m \frac{d \xi}{ds},\, -m \frac{d \eta}{ds},\, m \left( \eta \frac{d\xi}{ds} - \xi \frac{d \eta}{ds} \right),\, 1 \right)^T,\quad \bV_2 = \bbR \left(0,\, 0,\, -I_r,\, 0 \right)^T,\nonumber\\
\bbR = \begin{pmatrix}
\bbQ \bbA ^{-1} & 0 \\
0 & 1
\end{pmatrix}. \nonumber
\end{gather}
Here $\bz = (x,\, y,\, \alpha,\, s)^T$ is the vector in the expanded phase space $\mathcal{G}$. Note that not the coordinates of the internal mass, but its velocity of motion along the curve $\dot{s}$ and the angular velocity of the rotor $\Omega$  are taken as controls. The vector field $\bV _0$ corresponds to the drift, and the vector fields $\bV _1$, $\bV _2$ are associated with the control actions. We consider the vector fields
\begin{gather}
\begin{gathered}
\bV_1, \quad \bV_2, \quad \bV _3 = [\bV _1,\, \bV _2], \quad \bV _4 = [\bV _1,\, \bV _3], \quad \bV _5 = [\bV _2,\, \bV _3],\quad \bV _6 = [\bV _1,\, \bV_4]
\end{gathered} \label{full_vect}
\end{gather}
and show the completeness of their linear span. Here we deliberately do not consider the field $\bV_0$, since in the case of completeness of the linear span of the vector fields \eqref{full_vect} we also prove the controllability of motion for the case of zero circulation.

Consider separately two cases where the internal mass can move either along a straight line or along a circle. The choice of these curves is motivated by their simplicity.

1. We parameterize the reciprocating motion of the internal mass along the straight line as follows:
\begin{gather}
\xi = k_1 \sin s,\quad \eta = k_2 \sin s, \label{fb_law}
\end{gather}
where $k_1$, $k_2$ are constants that do not vanish simultaneously. In this case, the vector fields $\bV_1$, $\bV_2$, $\bV_3$, $\bV_5$ are dependent for
\begin{gather}
s = \pi / 2~\mbox{and}~s = 3\pi / 2. \label{criteria1_full}
\end{gather}
On the surfaces given by \eqref{criteria1_full}, the condition of linear dependence of the vector fields $\bV_1$, $\bV_2$, $\bV_4$, $\bV_6$ is
\begin{gather}
A \sin\alpha + B\cos\alpha = 0,\label{criteria2_full}
\end{gather}
where $A$ and $B$ are the coefficients of rather complicated form which depend only on the system parameters. Thus, the rank of the linear span of the vector fields \eqref{full_vect} is four everywhere in the phase space $\mathcal{G}$ except on the surface of codimension two, defined by \eqref{criteria1_full} and \eqref{criteria2_full}. Hence, the following theorem holds.
\begin{Theorem}\label{t2}
	An arbitrary body moving in a fluid (regardless of the presence of circulation around the body) with a given initial velocity can be moved by means of reciprocating motions of the internal mass and by rotating the internal rotor from any initial position to any final position with any initial and final positions of the internal mass.
\end{Theorem}

2. We parameterize the motion of the internal mass in a circle as follows:
\begin{gather}
\xi = r_o \cos s,\quad \eta = r_o \sin s, \quad r_o > 0.
\end{gather}
It is easy to show that in this case the vector fields $\bV_1$, $\bV_2$, $\bV_3$, $\bV_5$ are independent in the entire space $\mathcal{G}$. Hence, the following theorem holds.
\begin{Theorem}\label{t3}
	An arbitrary body moving in a fluid (regardless of the presence of circulation around the body) with a given initial velocity can be moved by means of the motion of the internal mass in a circle and by rotating the internal rotor from any initial position to any final position with any initial and final positions of the internal mass.
\end{Theorem}

\section{Stabilization of the body at a given point}

\subsection{Equations for controls}

Consider the stabilization of the body at a given point. Without loss of generality we assume that the system has started its motion from the origin of coordinates $x(0) = y(0) = 0$ with the initial orientation $\alpha(0) = 0$, the rotor and the movable mass being at rest. In this case the motion of the system occurs on the level set of the integrals
\begin{gather}
p_x = -\chi,\quad p_y = -\zeta,\quad K = 0,\quad F = \frac{\chi^2 + \zeta^2}{2 \lambda}.
\end{gather}

Suppose that over some interval the body moved under the control action and  came at time $t = T$ to the point $(x,\, y)$ with orientation $\alpha$. Let us formulate the problem of the body stabilization at this point as follows:

\textit{Can one choose limited control actions $\xi(t)$, $\eta(t)$ and $\Omega(t)$ at $t \geqslant T$ such that the body will stay arbitrarily long at point $(x,\, y)$ (possibly rotating about a fixed point).}

If a stabilization occurs without rotation ($\dot{\alpha} = 0$), we shall call it a complete stabilization, while a stabilization occurs with rotation about a fixed point ($\dot{\alpha} \neq 0$) will be called a partial stabilization. We impose the condition of limitation of control actions taking into account the possibility of their technical realization. In particular, the position of the moving mass is restricted by the boundary of the body, and the velocity of its motion and the rotor's rotation are restricted by the capabilities of the motors.

Since the controlled motion is described by the system of differential equations \eqref{kinem_1} in the expanded space, this problem may be viewed as a particular case of the problem of controlling \textit{a part} of variables. Therefore, the solution of this problem reduces not to the solution of some system of algebraic equations, but to analysis of differential equations governing the evolution of control actions $\xi(t)$, $\eta(t)$, $\Omega(t)$ and the remaining "free" variable $\alpha(t)$.

To define these equations, we substitute the equalities $\dot{x} = \dot{y} = 0$ into the first three equations of the system \eqref{kinem_1}. As a result, we obtain the system of differential equations in $\xi(t)$, $\eta(t)$, $\alpha(t)$ which contains an unknown function $\Omega(t)$
\begin{gather}
\bbD \frac{d}{dt} \begin{pmatrix}
\xi \\ \eta \\ \alpha
\end{pmatrix} = \begin{pmatrix}
\overline{x} \cos \alpha + \overline{y} \sin \alpha + \chi\\
-\overline{x} \sin \alpha + \overline{y} \cos\alpha + \zeta\\
\frac{\chi^2 + \zeta^2}{2\lambda} - \frac{\overline{x}^2+\overline{y}^2}{2\lambda}
\end{pmatrix} - \begin{pmatrix}
0 \\ 0 \\ I_r\Omega
\end{pmatrix}, \quad \bbD = \begin{pmatrix}
m  & 0 & -m\eta\\
0 & m & m\xi\\
-m\eta & m\xi & b
\end{pmatrix}.\label{stop}
\end{gather}
We recall that $\ol{x} = p_x - \lambda y$, $\ol{y} = p_y + \lambda x$. Making the change of variables
\begin{gather}
\xi = \rho \cos \left( \varphi - \ol{\varphi} \right), \quad \eta = \rho \sin \left( \varphi - \ol{\varphi} \right), \quad \psi = \alpha + \varphi - \ol{\varphi} + \ol{\psi}, \nonumber\\
\ol{\psi} = \arctan \frac{\ol{x}}{\ol{y}},\quad \ol{\varphi} = \arctan \frac{\chi}{\zeta},\nonumber
\end{gather}
we can represent equations \eqref{stop} as
\begin{gather}
\begin{gathered}
m\rho\dot{\psi} = r \cos \psi + \sigma \cos \varphi,\\
m \dot{\rho} = r \sin \psi + \sigma \sin \varphi,\\
(m\rho^2 + \ol{b}) \dot{\psi} - \ol{b} \dot{\varphi} = C_{xy} - I_r \Omega,
\end{gathered}\label{polar_eq}
\end{gather}
where we have used the notation
\begin{gather}
r = \sqrt{\ol{x}^2+\ol{y}^2},\quad \sigma = \sqrt{\chi^2 + \zeta^2},\quad C_{xy} = \frac{\sigma^2 - r^2}{2\lambda},\nonumber\\
\ol{b} = b - m \rho^2 = M (\xi_b^2 + \eta_b^2) + I + m_r (\xi_r^2 + \eta_r^2) + I_r + \lambda_6. \nonumber
\end{gather}
The system \eqref{polar_eq} contains three equations and four unknown functions $\rho$, $\psi$, $\varphi$, $\Omega$ and cannot be uniquely solved without using additional conditions. In what follows we consider several particular variations of the problem in which an answer to the question raised can be found.

\subsection{Complete stabilization of the body}

In the case of complete stabilization of the body we assume that the conditions $\dot{x} = \dot{y} = 0$, $\dot{\alpha}=0$ are satisfied at $t > T$. In this case, the following proposition holds.
\begin{Supp}\label{Supp.FullStop}
	A complete stabilization of the body during an infinite interval of time is possible if and only if the center of the body (point $O_1$ in Fig. \ref{frames}) is on the circle
	\begin{gather}
	x = \frac{\zeta}{\lambda} + \frac{\sqrt{\chi^2 + \zeta^2}}{\lambda} \cos \theta,\quad y = -\frac{\chi}{\lambda} + \frac{\sqrt{\chi^2 + \zeta^2}}{\lambda} \sin \theta,\quad \theta \in [0,\, 2\pi)\label{fixed}
	\end{gather}
	and its orientation is given by
	\begin{gather}
	\alpha = \theta - \arctan \frac{\zeta}{\chi} - \frac{\pi}{2}.\label{connection}
	\end{gather}
	The corresponding controls are equal to zero, i.e., the moving mass and the rotor are at rest.
\end{Supp}

\begin{proof} \textbf{Necessity.} Setting $\dot{\alpha} = 0$ in \eqref{stop}, we obtain the equations for the control actions
	\begin{gather}
	\begin{gathered}
	m \dot{\xi} = \ol{x} \cos \alpha + \ol{y} \sin \alpha + \chi,\\
	m \dot{\eta} = - \ol{x} \sin \alpha + \ol{y} \cos \alpha + \zeta,\\
	I_r \Omega +  m (\xi \dot{\eta} - \dot{\xi} \eta) = \frac{\chi^2 + \zeta^2}{2\lambda} - \frac{\ol{x}^2 + \ol{y}^2}{2\lambda}.
	\end{gathered}\label{full_stop}
	\end{gather}
	Since the right-hand sides of the first two equations \eqref{full_stop} are constant and in the general case are not equal to zero, $\xi$ and $\eta$ are linear functions of time. Hence, at a certain instant of time the internal mass will reach the boundary of the body and will have to stop. Thus, a stabilization during an infinite interval of time is possible only if the right-hand sides of the first two equations of \eqref{full_stop} are equal to zero. Setting the right-hand sides of \eqref{full_stop} to zero and expressing $x$ and $y$ from them, we obtain the equations for the circle \eqref{fixed} and the relation of the body's orientation to a point on the circle \eqref{connection}.
	
	By substituting the resulting coordinates of the points of stabilization into the third equation of \eqref{full_stop}, it is easy to check that all controls are equal to zero in this case.
	
	\textbf{Sufficiency.} Substituting \eqref{fixed} and \eqref{connection} into the integrals \eqref{exp_int}, we obtain a system of equations of the form
	\begin{gather}
	\bbA \bw + \bu = 0.
	\end{gather}
	The nondegeneracy of the matrix $\bbA$ implies that when there are no controls ($\bu=0$) and the body is at the points \eqref{fixed} with orientation \eqref{connection}, the body is at rest ($\bw = 0$).
\end{proof}

\begin{Note}
	For a complete stabilization of the body on the circle \eqref{fixed} the position of the internal mass is inessential.
\end{Note}

Since a complete stabilization at an arbitrary point is possible only during a finite interval of time, we consider the partial stabilization, i.e., $\dot{x} = \dot{y} = 0$, $\dot{\alpha} \neq 0$, by using \eqref{polar_eq}. In particular, we consider separately two cases in which additional restrictions (see Section 3) are imposed on the motion of the internal mass:
\begin{enumerate}
	\itemsep=-2pt
	\item partial stabilization by rotating the rotor and by moving the internal mass in a circle ($\rho = \const$).
	\item partial stabilization by rotating the rotor and by moving the internal mass along a straight line fixed in the body ($\varphi~=~\const$).
\end{enumerate}

\subsection{Partial stabilization at $\rho=\const$}

Consider the partial stabilization in the case where the internal mass can move only in the circle of a given radius. For $\rho = \ol{\rho} = \const$, equations \eqref{polar_eq} take the form
\begin{gather}
\begin{gathered}
m \ol{\rho} \dot{\psi} = r \cos \psi + \sigma \cos \varphi,\\
r \sin \psi + \sigma \sin \varphi = 0,\\
(m\ol{\rho}^2 + \ol{b}) \dot{\psi} - \ol{b} \dot{\varphi} = C_{xy} - I_r \Omega.
\end{gathered}\label{rho123}
\end{gather}
The first two equations of \eqref{rho123} form a closed system whose solution can yield the functions $\varphi(t)$ and $\psi(t)$. Then, using the known functions $\varphi(t)$ and $\psi(t)$ from the third equation of \eqref{rho123}, we can express the dependence $\Omega(t)$
\begin{gather}
\Omega = \frac{1}{I_r} \left( C_{xy} - \frac{1}{m\ol{\rho}} (r \cos \psi + \sigma \cos \varphi)   \left( m \ol{\rho}^2 + \ol{b} + \ol{b} \frac{r \cos \psi}{\sigma \cos \varphi} \right) \right).\label{omega_pp}
\end{gather}

\begin{Note}
	At the initial instant of time the functions $\varphi$ and $\psi$ must satisfy the second equation of \eqref{rho123}. This condition can be fulfilled by choosing the initial position of the internal mass. The possibility of choosing this initial position is ensured by Theorem \ref{t3}, from which it follows that we can bring the body to the end point with an arbitrary position of the moving mass.
\end{Note}

The solution of the system of equations \eqref{rho123} depends on the relationship between the parameters $r$ and $\sigma$. By a direct substitution it is easy to show that the condition $r = \sigma$ corresponds to the points of the circle \eqref{fixed}. This circle divides the plane $(x,\, y)$ into two regions (see Fig. \ref{no_dreif}), in which the solution of the system \eqref{rho123} is different. Moreover, in what follows we shall prove that the following proposition holds.

\begin{figure}[h!]
	\begin{center}
		\includegraphics[width=0.4\linewidth]{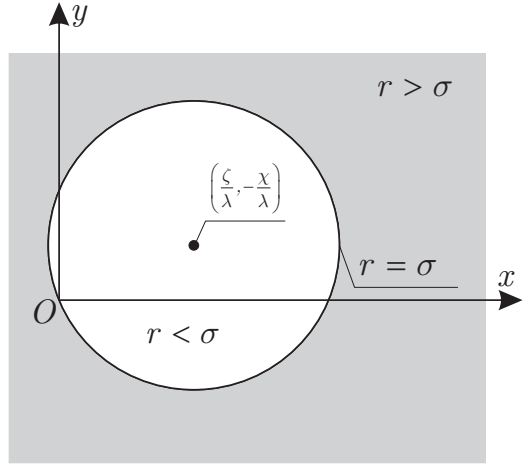}
		\caption{Regions of various solutions}\label{no_dreif}
		\vspace{-5mm}
	\end{center}
\end{figure}

\begin{Supp}\label{Supp.rhoConst}
	A partial stabilization during an infinite interval of time by rotating the rotor and by moving the internal mass in a circle is possible only inside and on the circle $r = \sigma$ (see Fig. \ref{no_dreif})
\end{Supp}

The proof of Proposition \ref{Supp.rhoConst} is given in Appendix А.

\subsection{Partial stabilization for $\varphi = \const$}

Consider a partial stabilization in the case where the internal mass can move only along a straight line fixed in the body. The controllability in this case was proved in Theorem \ref{t2}. In the first two equations of \eqref{polar_eq} we set $\varphi = \varphi _0 = \const$. Then they take the form
\begin{gather}
\begin{gathered}
m \rho \dot{\psi} = r \cos \psi + \sigma \cos \varphi _0, \\
m \dot{\rho} = r \sin \psi + \sigma \sin \varphi _0,
\end{gathered}\label{phi_eq1}
\end{gather}
Solving the system \eqref{phi_eq1}, we can obtain the functions $\varphi(t)$, $\rho(t)$. Then, using the known functions $\varphi(t)$ and $\rho(t)$ from the third equation of \eqref{polar_eq}, we can express the velocity of rotation of the rotor
\begin{gather}
\Omega(t) = \frac{1}{I_r} \left( C_{xy} - \frac{m \rho^2 + \ol{b} }{m\rho} (r \cos \psi + \sigma \cos \varphi _0) \right)\label{phi_eq3}.
\end{gather}

Analysis of equations \eqref{phi_eq1} and \eqref{phi_eq3} shows that the following proposition holds.

\begin{Supp}\label{Supp.phiConst}
	A partial stabilization during an infinite interval of time by rotating the rotor and by means of reciprocating motions of the internal mass in a circle is possible only inside and on the circle $r = \sigma$ (see Fig. \ref{no_dreif})
\end{Supp}

The proof of Proposition \ref{Supp.phiConst} is given in Appendix B.

\section{Conclusion}

The investigation has shown that the motion of a hydrodynamically asymmetric body in an ideal fluid in the presence of circulation around the body is completely controllable (in the sense of the Rashevskii-Chow theorem) by changing the position of the center of mass (the motion of the internal mass is a motion in a circle or a reciprocating motion) and the angular momentum of the system. Moreover, an arbitrary motion of the system can be performed only by means of an appropriate rotation of the internal rotor.

We have also considered  the possibility of compensation of drift by means of control. In particular, it was shown that by means of circular or reciprocating motion of the internal mass and by rotating the rotor the drift can be compensated for during an infinite interval of time if the body is inside some circular region. The center and the radius of this region are defined by the body geometry and by the amount of circulation of the velocity around the body. Outside this region, the drift can be compensated for using the above-mentioned patterns of motion of the internal mass only during a finite interval of time.

We list a number of open problems that must be solved to design real devices:
\vspace{-2mm}
\begin{enumerate}
	\itemsep=-2pt
	\item Construction of sufficiently simple and feasible patterns of motion of the internal mass to ensure a complete stabilization of the body at an arbitrary point of space during an infinite interval of time.
	\item Construction of explicit control to ensure the motion from one point of space to another.
	\item Motion control by variable circulation. For the case where circulation is a piecewise constant function of time, this problem has been solved in \cite{Ram_Ten_Tre}. In this case, preservation of first integrals of motion was used on the intervals of constant circulation. Of great interest is a more general problem, namely, that of constructing controls by changing circulation according to a smooth law, when the system admits no first integrals of motion.
\end{enumerate}

The authors thank A.V.Borisov and I.S.Mamaev for fruitful discussions.

The work of E.V. Vetchanin was supported by the RFBR grant 15-08-09093-a. The work of A.A. Kilin was supported by the RFBR grant 14-01-00395-a.

\section{Appendix A. Proof of Proposition 4.2}

\begin{proof} We break up the proof into three stages.
	
	1. For $r < \sigma$ the second equation of \eqref{rho123} has two nonintersecting solutions
	\begin{gather}
	\varphi = \begin{cases}
	- \arcsin \left( \dfrac{r}{\sigma} \sin \psi \right) \in \left[-\arcsin \dfrac{r}{\sigma},\, \arcsin \dfrac{r}{\sigma} \right] \subset \left[ - \dfrac{\pi}{2},\, \dfrac{\pi}{2} \right]\\
	\pi + \arcsin \left( \dfrac{r}{\sigma} \sin \psi \right) \in \left[\pi -\arcsin \dfrac{r}{\sigma},\, \pi + \arcsin \dfrac{r}{\sigma} \right] \subset \left[ \dfrac{\pi}{2},\, \dfrac{3\pi}{2} \right]
	\end{cases} \label{Phi_sol}
	\end{gather}
	where $\psi \in [-\pi,\,\pi]$. The realization of a specific branch of the solution \eqref{Phi_sol} depends on the initial value of $\varphi$, which is determined by the position of the internal mass at the instant of arrival at the point $(x,\, y)$. Moreover, according to Theorem \ref{t3} of complete controllability proved above,
	one can ensure, using appropriate controls, the realization of the required branch of the solution \eqref{Phi_sol} at the initial instant of time.
	
	Consider the branch $\varphi \in \left[-\arcsin \dfrac{r}{\sigma},\, \arcsin \dfrac{r}{\sigma} \right]$. Then $\cos \varphi > 0$, and the first equation of \eqref{rho123}, using the second equation, takes the form
	\begin{gather}
	m \ol{\rho} \dot{\psi} = r \cos \psi + \sigma \sqrt{1 - \dfrac{r^2}{\sigma^2} \sin ^2 \psi}.\label{Psi_eq}
	\end{gather}
	Equation \eqref{Psi_eq} has the following solution:
	\begin{gather}
	E \left( \psi,\, \frac{r}{\sigma} \right) - E \left( \psi _0,\, \frac{r}{\sigma} \right) - \frac{r}{\sigma} (\sin \psi - \sin \psi_0) = \frac{\sigma^2 - r^2}{m \ol{\rho} \sigma} t,
	\end{gather}
	where $E \left( \psi,\, \dfrac{r}{\sigma} \right)$ is the normal elliptic Legendre integral of the second kind. The right-hand side of \eqref{Psi_eq} is positive for any value of the angle $\psi$, hence, the function $\psi(t)$ increases monotonically. For the parameter values $\chi = 0.1$, $\zeta = 0.2$, $\lambda = 1.0$, $I_r = 1.0$, $m = 1.0$, $\ol{\rho} = 1.0$, $\ol{b} = 2.0$, $x = 0.1$, $y = 0.0$, $\varphi \in \left[-\arcsin \dfrac{r}{\sigma},\, \arcsin \dfrac{r}{\sigma} \right]$ the form of the functions $\psi$, $\varphi$, $\Omega$ is shown in Fig. \ref{internal_contol}.
	
	\begin{figure}[h!]
		\begin{center}
			\includegraphics[width=0.5\linewidth]{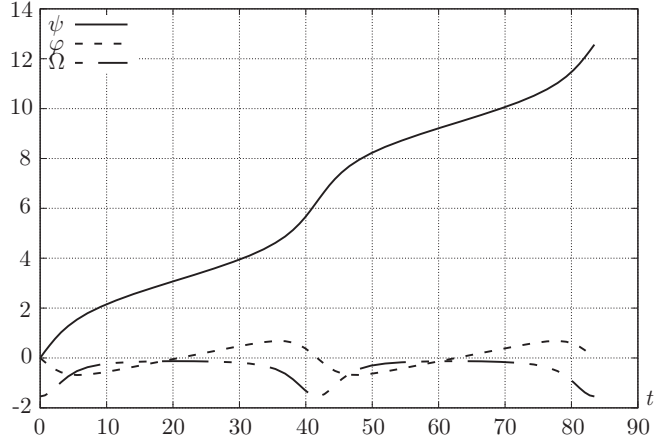}
			\caption{Form of the functions $\psi$, $\varphi$, $\Omega$ for the parameter values $\chi = 0.1$, $\zeta = 0.2$, $\lambda = 1.0$, $I_r = 1.0$, $m = 1.0$, $\ol{\rho} = 1.0$, $\ol{b} = 2.0$, $x = 0.1$, $y = 0.0$}\label{internal_contol}
			\vspace{-5mm}
		\end{center}
	\end{figure}
	The constructed control is periodic, restricted and ensures a partial stabilization during an arbitrarily long interval of time.
	
	Consider the second branch of the solution \eqref{Phi_sol} $\varphi \in \left[\pi -\arcsin \dfrac{r}{\sigma},\, \pi + \arcsin \dfrac{r}{\sigma} \right]$. The differential equation for the determination of $\psi$ has the form
	\begin{gather}
	m \ol{\rho} \dot{\psi} = r \cos \psi - \sigma \sqrt{1 - \dfrac{r^2}{\sigma^2} \sin ^2 \psi}.
	\end{gather}
	Its solution is expressed, just as for the first branch, in terms of the normal elliptic Legendre integral of the second kind:
	\begin{gather}
	E \left( \psi,\, \frac{r}{\sigma} \right) - E \left( \psi _0,\, \frac{r}{\sigma} \right) + \frac{r}{\sigma} (\sin \psi - \sin \psi_0) = - \frac{\sigma^2 - r^2}{m \ol{\rho} \sigma} t.
	\end{gather}
	A straightforward calculation shows that the controls corresponding to the second branch are restricted during an arbitrarily long interval of time also.
	
	2. For $r = \sigma$ the second equation of \eqref{rho123} has two solutions
	\begin{gather}
	\varphi = \begin{cases}
	\psi + \pi.\\
	- \psi
	\end{cases}
	\end{gather}
	
	If the equality $\varphi = \psi + \pi$ holds at the instant of arrival at a given point of space, then the first equation of \eqref{rho123} takes the form
	\begin{gather}
	m \ol{\rho} \dot{\psi} = 0
	\end{gather}
	i.e., the point $\varphi = \psi + \pi$ is a fixed point of the system. By straightforward calculations one can readily verify that this case corresponds to the solution \eqref{fixed} with $\alpha = \ol{\varphi} - \ol{\psi} - \pi$.
	
	If the equality $\varphi = -\psi$ holds at the instant of arrival at a given point of space, then the first equation of \eqref{rho123} takes the form
	\begin{gather}
	m \ol{\rho} \dot{\psi} = 2 r \cos \psi \label{eq111}
	\end{gather}
	
	Equation \eqref{eq111} has two steady-state solutions: a stable one, $\psi _1 = \dfrac{\pi}{2}$, and an unstable one, $\psi _2 = \dfrac{3\pi}{2}$. The general solution of this equation has the form
	\begin{gather}
	\frac{m \ol{\rho}}{4 r} \ln \left\vert \frac{1+\sin \psi}{1 - \sin \psi} \right\vert = t + C
	\end{gather}
	where $C$ is the constant of integration. It is clear from the form of the general solution that the approach to the point $\psi _1$ occurs in infinite time.
	
	The rotational velocity of the rotor can be calculated from the third equation of \eqref{rho123} and takes the form
	\begin{gather}
	\Omega = - \frac{m \ol{\rho}^2 + 2 \ol{b}}{m \ol{\rho} I_r} 2 r \cos \psi,
	\end{gather}
	whence it is clear that the value of $\Omega$ is finite. Thus, for $r = \sigma$ a stabilization is possible in infinite time.
	
	3. For $r > \sigma$ it is more convenient to express $\psi$ from the second equation of \eqref{rho123} as follows:
	\begin{gather}
	\psi = \begin{cases}
	- \arcsin \left( \dfrac{\sigma}{r} \sin \varphi \right) \in \left[-\arcsin \dfrac{\sigma}{r},\, \arcsin \dfrac{\sigma}{r} \right] \subset \left[ - \dfrac{\pi}{2},\, \dfrac{\pi}{2} \right]\\
	\pi + \arcsin \left( \dfrac{\sigma}{r} \sin \varphi \right) \in \left[\pi -\arcsin \dfrac{\sigma}{r},\, \pi + \arcsin \dfrac{\sigma}{r} \right] \subset \left[ \dfrac{\pi}{2},\, \dfrac{3\pi}{2} \right]
	\end{cases}\label{Psi_sol}
	\end{gather}
	where $\varphi \in [-\pi,\,\pi]$. The realization of a specific branch of the solution \eqref{Psi_sol} depends on the initial value of $\psi$, which is defined by the position of the internal mass and by the orientation of the body at the instant of arrival at the point $(x,\, y)$. Moreover, according to Theorem \ref{t3} of complete controllability proved above, using a suitable control one can ensure the realization of the required branch of the solution \eqref{Psi_sol} at the initial instant of time. Consider the branch $\psi \in \left[-\arcsin \dfrac{\sigma}{r},\, \arcsin \dfrac{\sigma}{r} \right]$ of the solution \eqref{Psi_sol}. This branch corresponds to the inequality $\cos\psi > 0$, and the differential equation for the determination of $\varphi$ is obtained from the second equation of \eqref{rho123} and has the form
	\begin{gather}
	\left( \frac{\sigma^2 \cos ^2 \varphi}{\sqrt{r^2 - \sigma^2 \sin^2 \varphi}} - \sigma \cos \varphi \right) \dot{\varphi} = \frac{r^2 - \sigma^2}{m \ol{\rho}}.\label{eq_Phi1}
	\end{gather}
	Let us examine the phase trajectories of \eqref{eq_Phi1}. To do so, we express $\dot{\varphi}$ as follows:
	\begin{gather}
	\dot{\varphi} = \frac{r^2 - \sigma^2}{2 \ol{\rho}} \cdot \frac{\sqrt{r^2 - \sigma^2 \sin ^2 \varphi}}{\sigma \cos \varphi (\sigma \cos \varphi - \sqrt{r^2 - \sigma^2 \sin ^2 \varphi})}\label{phi_phase}
	\end{gather}
	In the case at hand, $\sigma \cos \varphi - \sqrt{r^2 - \sigma^2 \sin ^2 \varphi} < 0$ always holds. Hence, $\dot{\varphi}(\varphi)$ undergoes a discontinuity of the second kind at the points $\varphi = \pm \dfrac{\pi}{2}$:
	\begin{gather}
	\lim _{\varphi \rightarrow \pm \frac{\pi}{2} \mp 0} \dot{\varphi} = -\infty,\quad \lim _{\varphi \rightarrow \mp \frac{\pi}{2} \pm 0} \dot{\varphi} = +\infty.
	\end{gather}
	Note that these singularities do not depend on the value of $\dfrac{r}{\sigma}$. We also note that the function \eqref{phi_phase} does not vanish, and hence the system \eqref{phi_phase} has no fixed points.
	
	The phase trajectories of \eqref{eq_Phi1} for $\dfrac{r}{\sigma} = 1.1$ and various values of $\beta~=~\dfrac{r^2 - \sigma^2}{m \ol{\rho} \sigma}$ are shown in Fig. \ref{phase_traj}
	\begin{center}
		\includegraphics[width=0.5\linewidth]{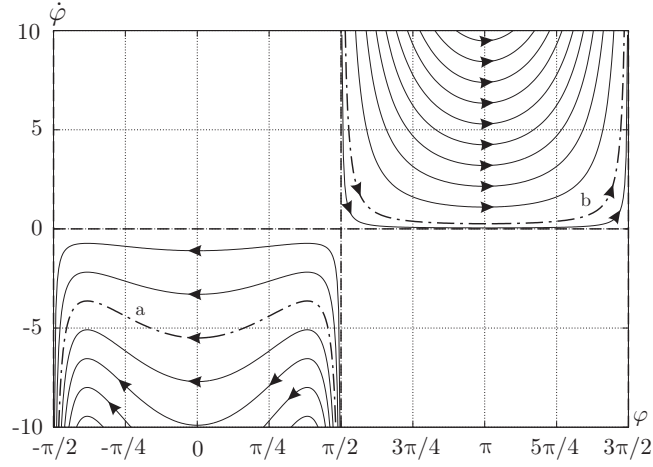}
		\captionof{figure}{Phase trajectories of the system \eqref{phi_phase}. The trajectories shown in the figure correspond to $\beta~=~0.5$ and to various motion patterns.}\label{phase_traj}
	\end{center}
	
	Depending on the initial conditions, two motion patterns are possible for the same value of $\beta$. The corresponding functions $\varphi(t)$ and $\Omega(t)$ are shown in Figs. \ref{PhiOmega}.
	
	\begin{figure}[h!]
		\begin{center}
			\includegraphics[width=0.48\linewidth]{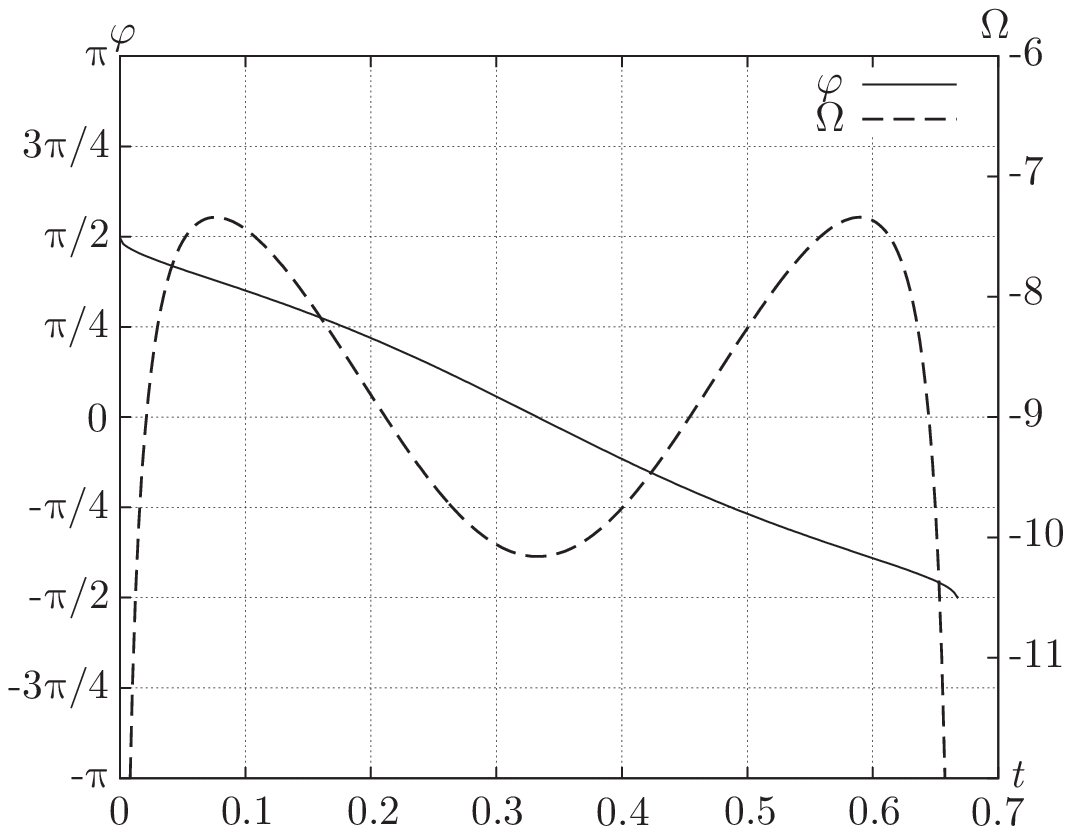}
			\includegraphics[width=0.48\linewidth]{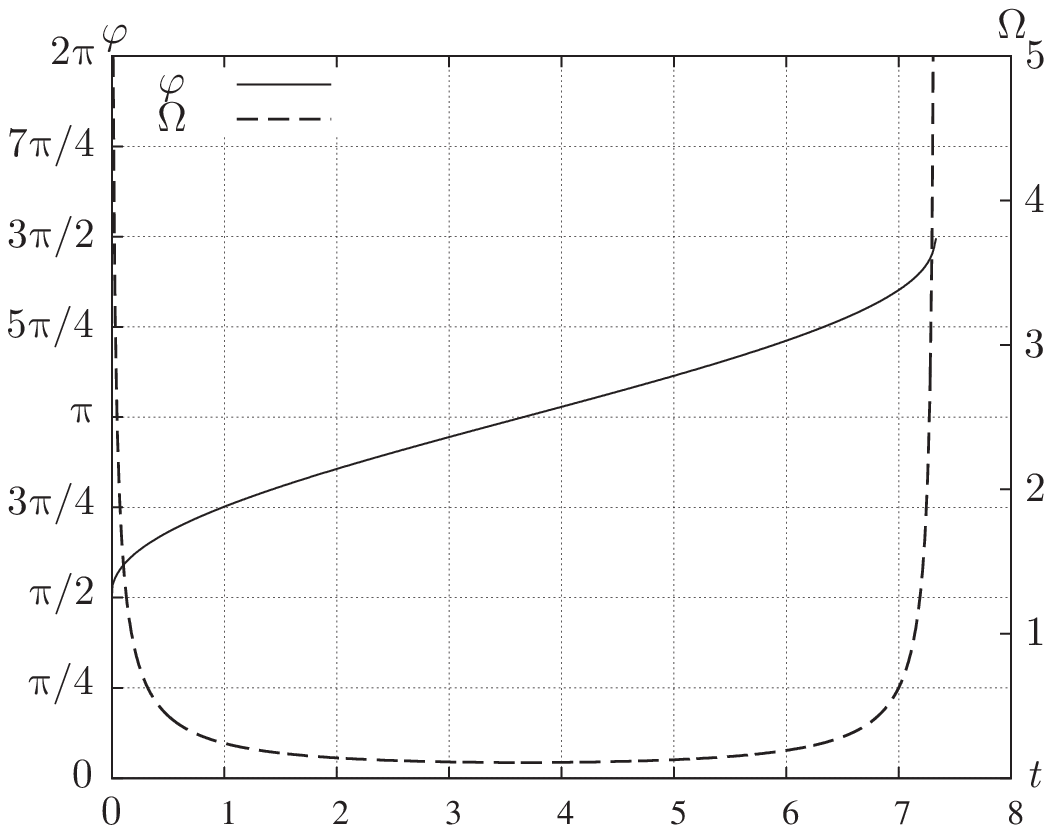}\\
			\begin{tabular}{p{0.48\linewidth}p{0.48\linewidth}}
				\parbox[c]{\linewidth}{\centering a)} & \parbox[c]{\linewidth}{\centering b)}
			\end{tabular}
			\caption{Functions $\varphi(t)$ and $\Omega(t)$ corresponding to different branches of the solution \eqref{Phi_sol} for $\beta~=~0.5$. The graphs correspond to the phase trajectories a) and b) in the previous figure.}\label{PhiOmega}
		\end{center}
	\end{figure}
	
	It can be seen from Figs. \ref{PhiOmega} that the function $\varphi$ reaches the critical values $-\dfrac{\pi}{2}$ and $\dfrac{3\pi}{2}$ in finite time. Using \eqref{omega_pp}, it is easy to check that $\Omega$ increases infinitely as $\varphi \rightarrow \pm \dfrac{\pi}{2}$. Hence, a partial stabilization is possible only in finite time.
	
	\begin{Note}
		Equation \eqref{eq_Phi1} has the following solution:
		\begin{gather}
		\begin{split}
		\frac{\sigma^2 - r^2}{r} \left( F \left(\varphi, \frac{\sigma}{r} \right) - F\left(\varphi _0, \frac{\sigma}{r} \right) \right) + {} & {} r \left( E \left(\varphi, \frac{\sigma}{r} \right) - E \left(\varphi _0, \frac{\sigma}{r} \right) \right) - \\
		{} - {} & {} \sigma (\sin \varphi - \sin \varphi_0)= \frac{r^2 - \sigma^2}{m \ol{\rho}} t,
		\end{split}\label{Phi_sol1}
		\end{gather}
		where $F \left( \varphi,\, \dfrac{\sigma}{r} \right)$ is the normal elliptic Legendre integral of the first kind. The above solution \eqref{Phi_sol1} includes two motion patterns corresponding to different initial conditions  $\varphi \in \left(-\dfrac{\pi}{2},\, \dfrac{\pi}{2} \right)$ and $\varphi \in \left(\dfrac{\pi}{2},\, \dfrac{3\pi}{2} \right)$.
	\end{Note}
	
	A straightforward calculation shows that the solution corresponding to the branch $\psi \in \left[\pi -\arcsin \dfrac{\sigma}{r},\, \pi + \arcsin \dfrac{\sigma}{r} \right]$ behaves similarly. In this case, the equation for the determination of $\psi$ and its solution have the form
	\begin{gather}
	\left( \frac{\sigma^2 \cos ^2 \varphi}{\sqrt{r^2 - \sigma^2 \sin^2 \varphi}} + \sigma \cos \varphi \right) \dot{\varphi} = \frac{\sigma^2- r^2}{m \ol{\rho}},\\
	\begin{split}
	\frac{\sigma^2 - r^2}{r} \left( F \left(\varphi, \frac{\sigma}{r} \right) - F\left(\varphi _0, \frac{\sigma}{r} \right) \right) + {} & {} r \left( E \left(\varphi, \frac{\sigma}{r} \right) - E \left(\varphi _0, \frac{\sigma}{r} \right) \right) + \\ {} + {} & {} \sigma (\sin \varphi - \sin \varphi_0)= \frac{\sigma^2 - r^2}{m \ol{\rho}} t.
	\end{split}
	\end{gather}
\end{proof}

\section{Appendix B. Proof of Proposition 4.3}

\begin{proof} First of all, we examine the general properties of the system of equations \eqref{phi_eq1}--\eqref{phi_eq3}. It is easy to verify that equations \eqref{phi_eq1} possess the symmetry
	\begin{gather}
	\rho \rightarrow \rho,\quad \psi \rightarrow - \psi,\quad \varphi _0 \rightarrow - \varphi _0,\quad \Omega \rightarrow \Omega,\quad t \rightarrow -t \label{symmetry}
	\end{gather}
	and the integral of motion
	\begin{gather}
	G = \rho e ^ {- \mathcal{F} (\psi)} = \const, \quad
	\mathcal{F} (\psi) = \int \frac{r \sin \psi + \sigma \sin \varphi _0}{r \cos \psi + \sigma \cos \varphi _0} d\psi. \label{rho_psi_int}
	\end{gather}
	The integral \eqref{rho_psi_int} and hence the behavior of the system depend on three parameters $r$, $\sigma$, and $\varphi _0$. The parameter $\varphi _0$ is related to the direction of motion of the internal mass, to circulation and the body geometry.
	
	Below we consider separately several cases depending on the values of these parameters. 	
	
	
	1. The condition $\sin \varphi _0 = 0$ corresponds to two values: $\varphi _0 = 0$ and $\varphi _0 = \pi$, and the integral \eqref{rho_psi_int} takes the form
	\begin{gather}
	G = \rho (r \cos \psi \pm \sigma). \label{rho11}
	\end{gather}
	Here the sign $+$ corresponds to $\varphi _0 = 0$, and the sign $-$ corresponds to $\varphi _0 = \pi$. In view of \eqref{rho11} the first equation of \eqref{phi_eq1} takes the form
	\begin{gather}
	\dot{\psi} = \frac{1}{mG} (r \cos \psi \pm \sigma )^2. \label{Psi11}
	\end{gather}
	Its solution depends on the relationship between the parameters $r$ and $\sigma$.
	
	1.1. For $r < \sigma$ (the center of the body is inside the circle $r=\sigma$), according to \eqref{rho11}, the function $\rho (t)$ has no singularities, is periodic and continuous for any value of $\psi$ and hence bounded on a given level set of the integral $G$. The right-hand side of \eqref{Psi11} preserves the sign and never vanishes, hence, the function $\psi (t)$ is monotonous and equation \eqref{Psi11} has no fixed points. An example of the functions $\rho (t)$, $\psi (t)$ and $\Omega (t)$ for $\varphi_0 = 0$ is shown in Fig. \ref{internal_fb1}. Thus, for $\sin \varphi _0 = 0$ and $r < \sigma$ a partial stabilization is possible during an arbitrarily long interval of time.
	
	\begin{figure}[h!]
		\begin{center}
			\includegraphics[width=0.48\linewidth]{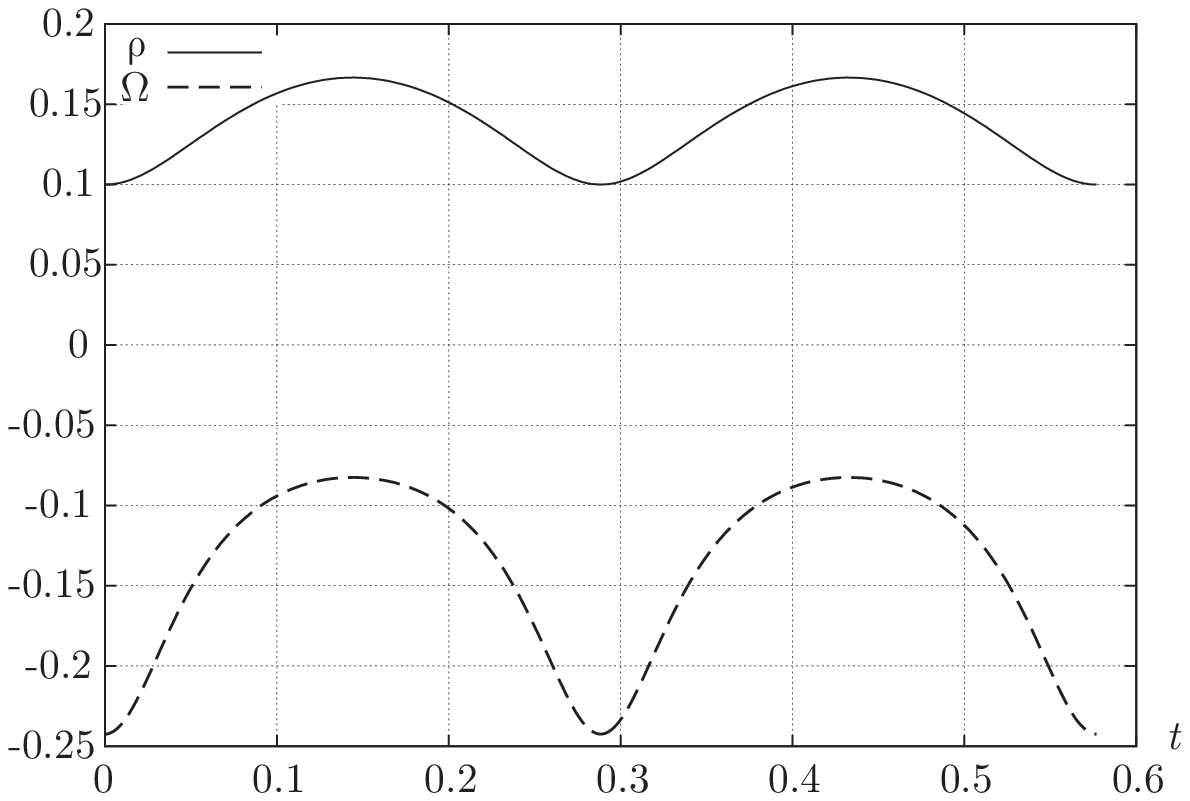}
			\includegraphics[width=0.48\linewidth]{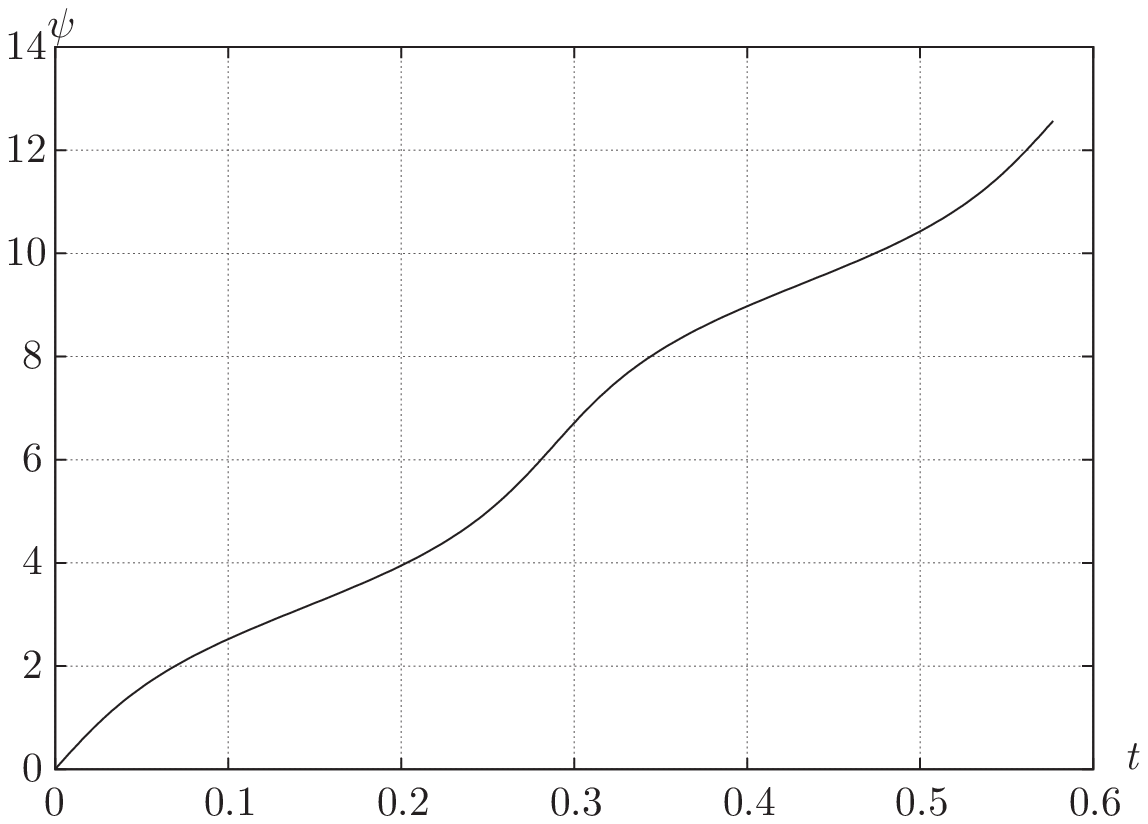}
			\caption{Functions $\rho (t)$, $\psi (t)$ and $\Omega (t)$}\label{internal_fb1}
		\end{center}
		\vspace{-8mm}
	\end{figure}
	
	1.2. For $r = \sigma$ (the center of the body is on the circle $r=\sigma$) the system of equations \eqref{phi_eq1} has a family of fixed points lying on the straight line $\psi = \psi_* = \pi + \varphi _0$. The phase trajectories of the system for $\varphi _0 = \pi$ and various values of the integral $G$ are shown in Fig. \ref{Phase_port10}.
	
	\begin{figure}[h!]
		\begin{center}
			\includegraphics[width=0.5\linewidth]{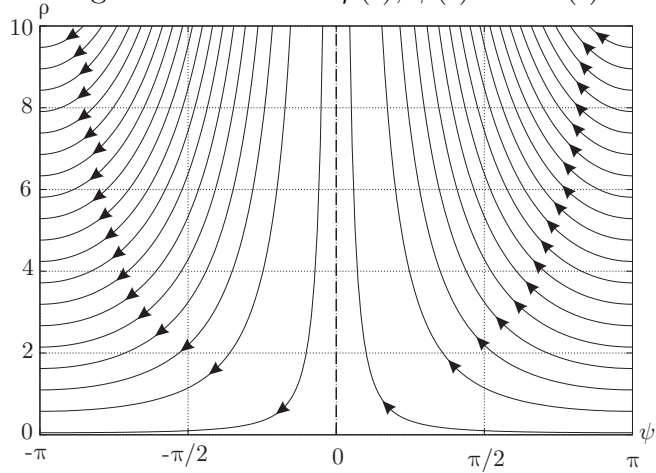}
			\caption{Phase trajectories of the system \eqref{phi_eq1} for $\varphi _0 = \pi$ and $\dfrac{r}{\sigma} = 1$}\label{Phase_port10}
		\end{center}
		\vspace{-5mm}
	\end{figure}
	\noindent It can be seen from Fig. \ref{Phase_port10} that the phase variable $\rho$ increases infinitely in a neighborhood of the straight line $\psi = \psi _*$.
	
	Let us examine the attainability of a fixed point. To do so, we
	linearize equation \eqref{Psi11} in its neighborhood
	\begin{gather}
	\dot{\psi} = \frac{r^2 \sin ^2  \psi _*}{m G} (\psi - \psi _*)^2.\label{ser_eq}
	\end{gather}
	
	Let us integrate equation \eqref{ser_eq} on the interval $[\psi _* - \varepsilon,\, \psi _*)$
	\begin{gather}
	t \frac{r^2 \sin^2\psi_*}{mG} = \frac{1}{\varepsilon} - \lim _{\psi \rightarrow \psi_*}\frac{1}{\psi - \psi_*} = \infty.
	\end{gather}
	Consequently, the phase trajectories approach the straight line $\psi = \psi _*$ in infinite time. Thus, for $\sin \varphi _0 = 0$ and $r = \sigma$ a partial stabilization can be performed only in finite time.
	
	1.3. For $r > \sigma$ (the center of the body is outside the circle $r=\sigma$), equation \eqref{Psi11} admits particular steady-state solutions
	\begin{gather}
	\psi ^*_\pm = \begin{cases}
	\pm \left( \pi - \arccos \frac{\sigma}{r} \right), & \varphi_0 = 0,\\
	\pm \arccos \frac{\sigma}{r}, & \varphi _0 = \pi.
	\end{cases}\label{asymp}
	\end{gather}
	The phase trajectories of the system on the plane $(\rho,\, \psi)$ for various values of the integral $G$ and the parameter values $r=1$, $\sigma = 0.5$ are shown in Fig.\ref{Phase_port11}.
	
	\begin{figure}[h!]
		\begin{center}
			\includegraphics[width=0.5\linewidth]{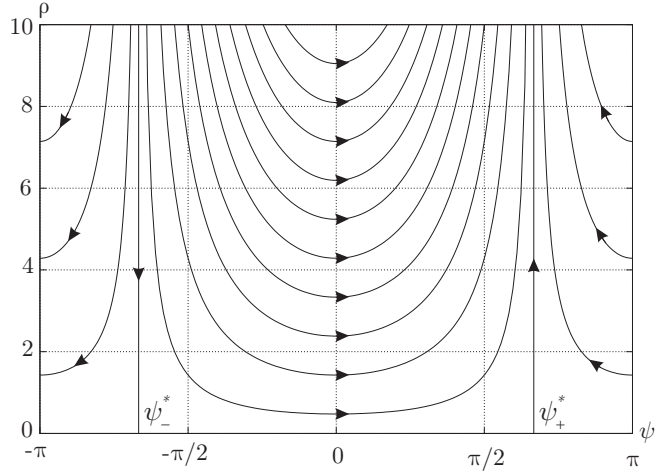}
			\caption{Phase trajectories of the system \eqref{phi_eq1} for $\sin \varphi _0 = 0$ and $\dfrac{r}{\sigma} > 1$}\label{Phase_port11}
		\end{center}
		\vspace{-5mm}
	\end{figure}
	
	It can be seen from Fig.\ref{Phase_port11} that the phase trajectories approach the vertical asymptotes $\psi = \psi ^*_\pm$, hence, as time goes on, $\rho \rightarrow +\infty$. Performing the same analysis as in the previous case, we can show that the value $\psi = \psi _{+}$ is reached in infinite time. Thus, for $\sin \varphi _0 = 0$ and $r > \sigma$ a partial stabilization can be performed only in finite time.
	
	2. Consider a more general case for which the line of motion of the internal mass is such that $\sin \varphi _0 \neq 0$. In this case, the integral \eqref{rho_psi_int} can be written as
	\begin{gather}
	G = \rho (r \cos \psi + \ol{\sigma} ) \exp \left( - \sigma \sin \varphi _0 \int \frac{d\psi}{r \cos \psi + \ol{\sigma}} \right)=\const,\quad \ol{\sigma} = \sigma \cos \varphi _0.\label{arb_Phi0}
	\end{gather}
	The exact form of the integral \eqref{arb_Phi0} depends on the relationship between $r$ and $\sigma \vert \cos \varphi _0 \vert$. The equality $r = \sigma \vert \cos \varphi _0 \vert$ defines the circle with the center at the point $\left(\dfrac{\zeta}{\lambda},\, -\dfrac{\chi}{\lambda} \right)$ and radius $\dfrac{\sqrt{\chi^2 + \zeta^2}}{\lambda} \vert \cos \varphi _0 \vert$.
	
	2.1. For $r < \sigma \vert \cos \varphi _0 \vert$ (the center of the body is inside the circle $r = \sigma \vert \cos \varphi _0 \vert$) the integral \eqref{arb_Phi0} is not unique and can be represented as
	\begin{multline}
		\ol{G} = \frac{G}{r} = \rho (\cos \psi + \varkappa ) \times \\
		\times \exp \left( - \frac{2 \varkappa \tan \varphi _0}{\sqrt{\varkappa^2 - 1}}  \left( \arctan \left( \frac{\sqrt{\varkappa^2 - 1}}{\varkappa + 1} \tan \frac{\psi}{2}\right) + \left[ \frac{\psi + \pi}{2\pi}\right] \pi \right) \right)=\const,\label{arb_Phi0_121}
	\end{multline}
	where $\varkappa = \dfrac{\ol{\sigma}}{r}$, $\vert\varkappa\vert  > 1$, $\psi \in [-\pi,\, \pi)$.
	
	The trajectories of the system \eqref{phi_eq1} fill everywhere densely the plane $(\rho,\, \psi)$. Depending on the relationship between the parameters $r$, $\sigma$, and $\varphi _0$, two types of phase portraits are possible (see Fig. \ref{Phase_port12}).
	
	\begin{figure}[h!]
		\begin{center}
			\includegraphics[width=0.48\linewidth]{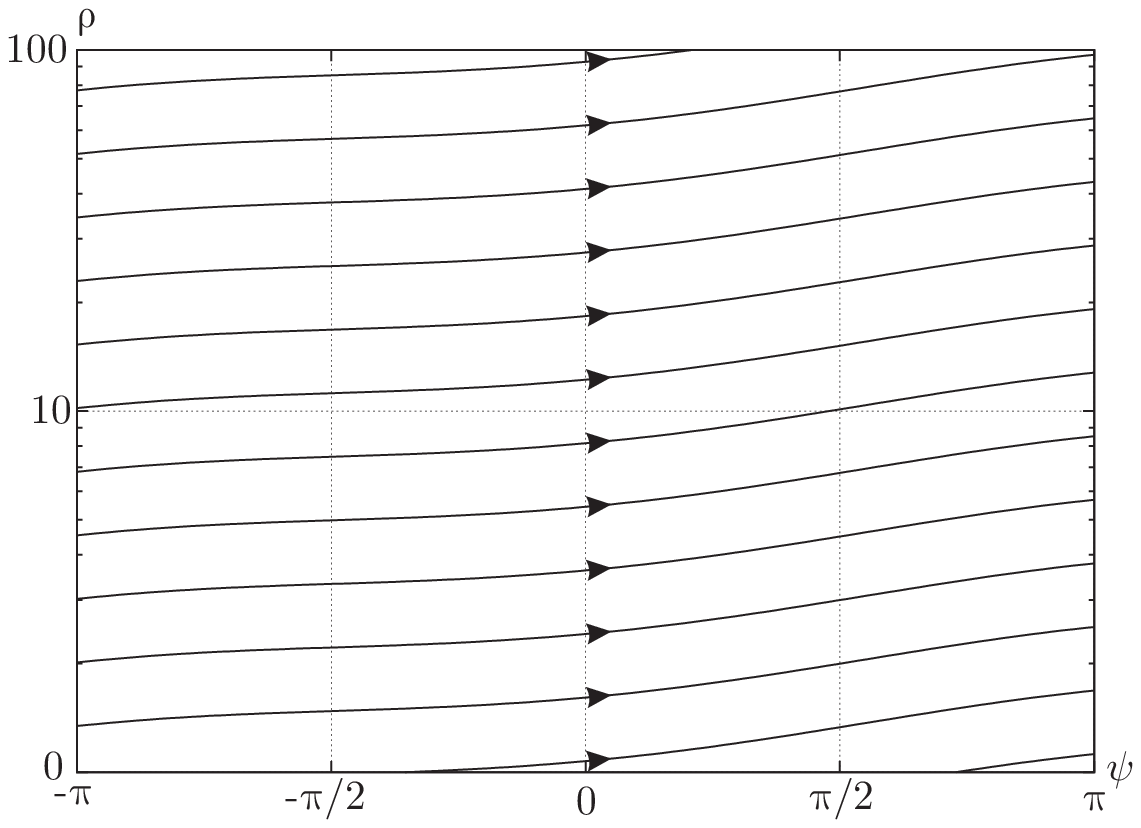}
			\includegraphics[width=0.48\linewidth]{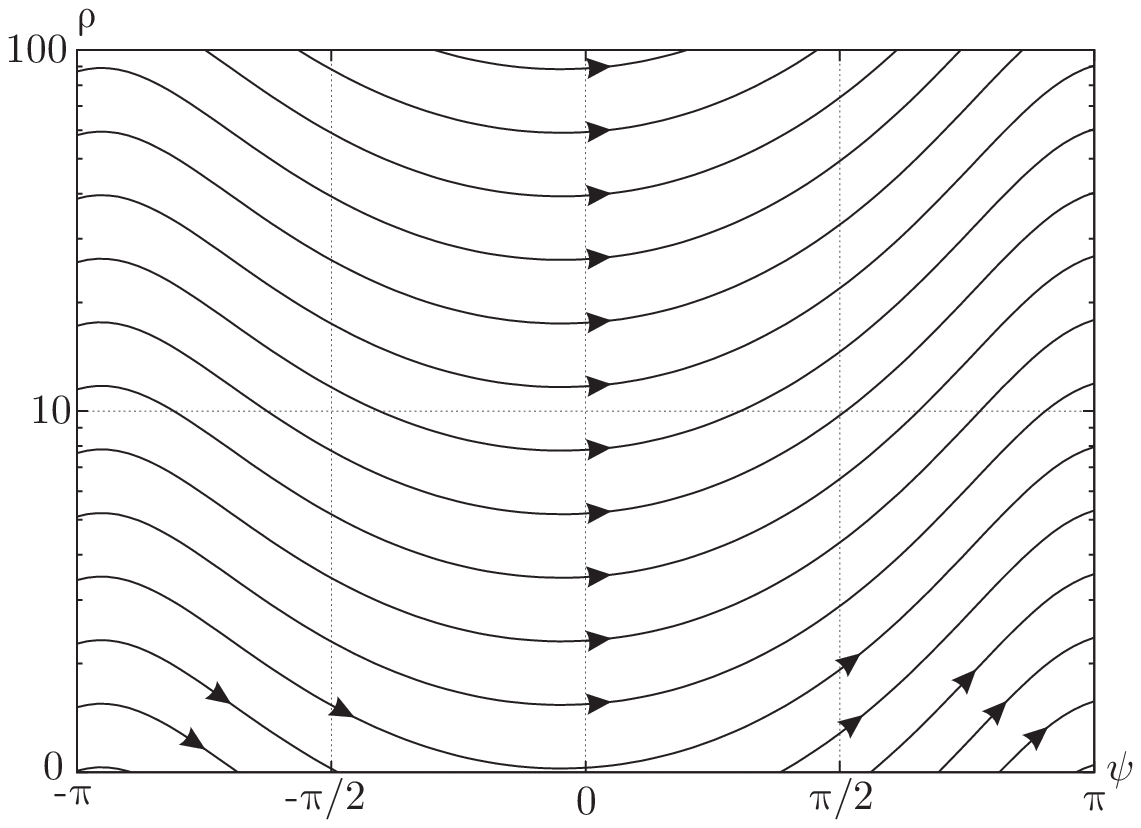}\\
			a) \hspace{8cm} b)
			\caption{Phase portraits of the system for $\sigma = 1.5$, $\varphi _0 = 0.1$. a) $r = 0.1$, b) $r = 1$} \label{Phase_port12}
		\end{center}
		\vspace{-5mm}
	\end{figure}
	
	Indeed, the right-hand side of the second equation of \eqref{phi_eq1} is nonnegative (nonpositive) for $\dfrac{\sigma \vert \sin \varphi _0 \vert}{r} \geqslant 1$. This means nondecrease (nonincrease) of the function $\rho$ (see Fig.~\ref{Phase_port12}a). Otherwise the sign $\dot{\rho}$ changes twice in one period of the variable $\psi$, and the system trajectories have extrema (see Fig.~\ref{Phase_port12}b).
	
	According to the first equation of \eqref{phi_eq1}, the function $\psi(t)$ is monotonous, since the right-hand side of the equation is sign-definite by virtue of the condition $r < \sigma \vert \cos \psi _0 \vert$. Using the integral \eqref{arb_Phi0_121}, we estimate the change of $\rho$ for one period of the variable $\psi \in [-\pi,\, \pi)$
	\begin{gather}
	\Delta \rho = \rho _0 \left( \exp \left( \frac{2 \pi \varkappa \tan \varphi _0}{\sqrt{\varkappa^2 - 1}} \right) - 1 \right). \label{delta_rho}
	\end{gather}
	It can be seen from \eqref{delta_rho} that the increment $\Delta \rho$ of the phase variable $\rho$ is directly proportional to the value $\rho _0 = \rho \big\vert _{\psi=-\pi}$. Moreover, $\sign \Delta \rho = \sign \tan \varphi _0$. It is easy to show that for $N$ periods the increment is
	\begin{gather}
	\Delta \rho _N = \rho _0 \left( \exp \left( \frac{2 N \pi \varkappa \tan \varphi _0}{\sqrt{\varkappa^2 - 1}} \right) - 1 \right).
	\end{gather}
	That is, the increment depends exponentially on the number of periods $N$. Thus, despite the existence of two types of phase portraits, the phase variable $\rho$ increases on an average if $\tan \varphi _0 > 0$ and decreases on an average if $\tan \varphi _0 < 0$.
	
	According to \eqref{phi_eq3}, as $\rho$ decreases infinitely, $\Omega \rightarrow \infty$. Thus, if the condition $r < \sigma \vert \cos \varphi _0 \vert$ is satisfied, either $\rho$ or $\Omega$ increases indefinitely, depending on the value of $\varphi _0$. Thus, for $\sin \varphi _0 \neq 0$ and $r = \sigma \vert \cos \varphi _0 \vert$ a partial stabilization can be performed only in finite time.
	
	2.2. For $r > \sigma \vert \cos \varphi _0 \vert$ (the center of the body is outside the circle $r = \sigma \vert \cos \varphi _0 \vert$) the integral \eqref{arb_Phi0} can be written as
	
	\begin{gather}
	\ol{G} = - \rho \frac{\sign (\tau_{+}(\psi) \tau_{-}(\psi) )}{1 + \tan ^2 \frac{\psi}{2}} \vert \tau _{-} (\psi) \vert ^{\delta + 1} \vert \tau _{+} (\psi) \vert ^{1-\delta},\label{int_rhoPsi}\\
	\delta = \frac{\varkappa \tan \varphi_0}{\sqrt{1 - \varkappa^2}},\quad
	\tau _{\pm}(\psi) = \sqrt{ 1 - \varkappa } \tan \frac{\psi}{2} \pm \sqrt{ 1 + \varkappa }\nonumber.
	\end{gather}
	
	Consider the values $\psi _\pm = \mp 2 \arctan \sqrt{\frac{1 + \varkappa}{1 - \varkappa}}$, which are zeros of the functions $\tau _\pm(\psi)$. It is seen from~\eqref{int_rhoPsi} that the behavior of the system \eqref{phi_eq1} in a neighborhood of the lines $\psi = \psi _\pm$ can change depending on the parameter $\delta$. For the values $\delta < 0$ three possible types of phase portraits are shown in Fig. \ref{Phase_port13}.
	
	\begin{figure}[h!]
		\begin{center}
			\includegraphics[width=0.32\linewidth]{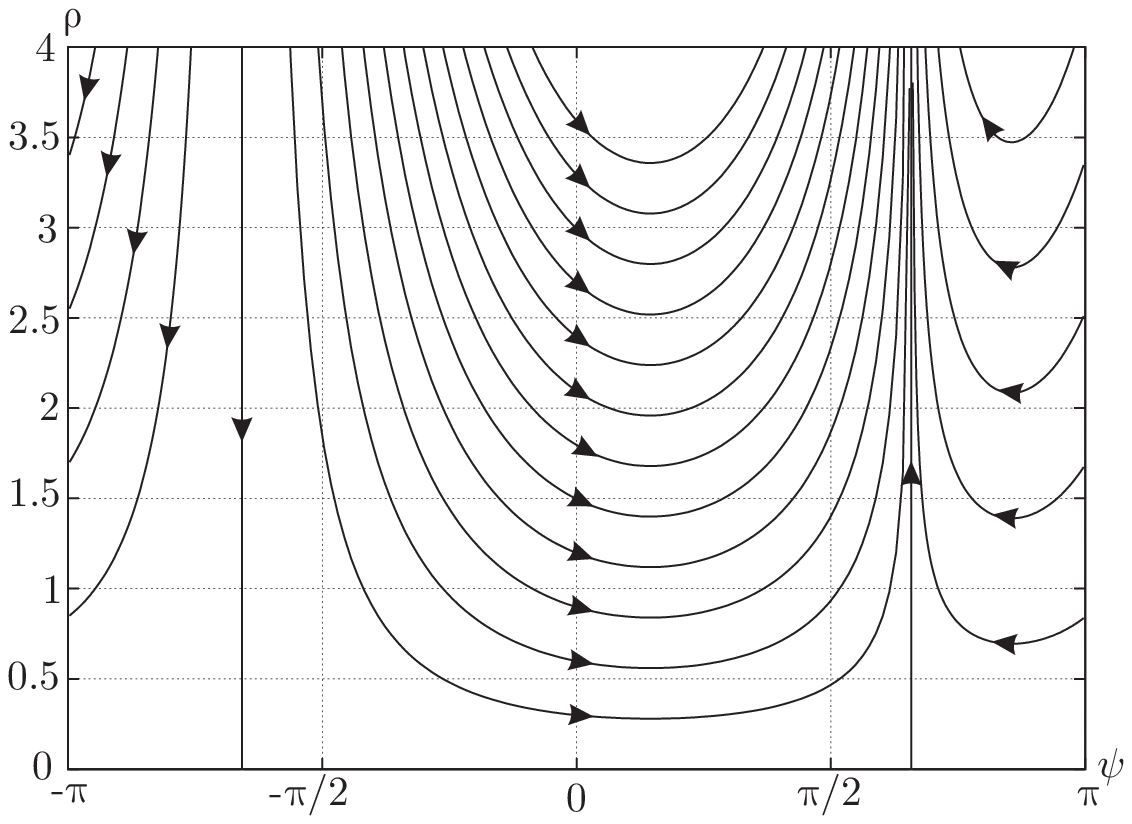}
			\includegraphics[width=0.32\linewidth]{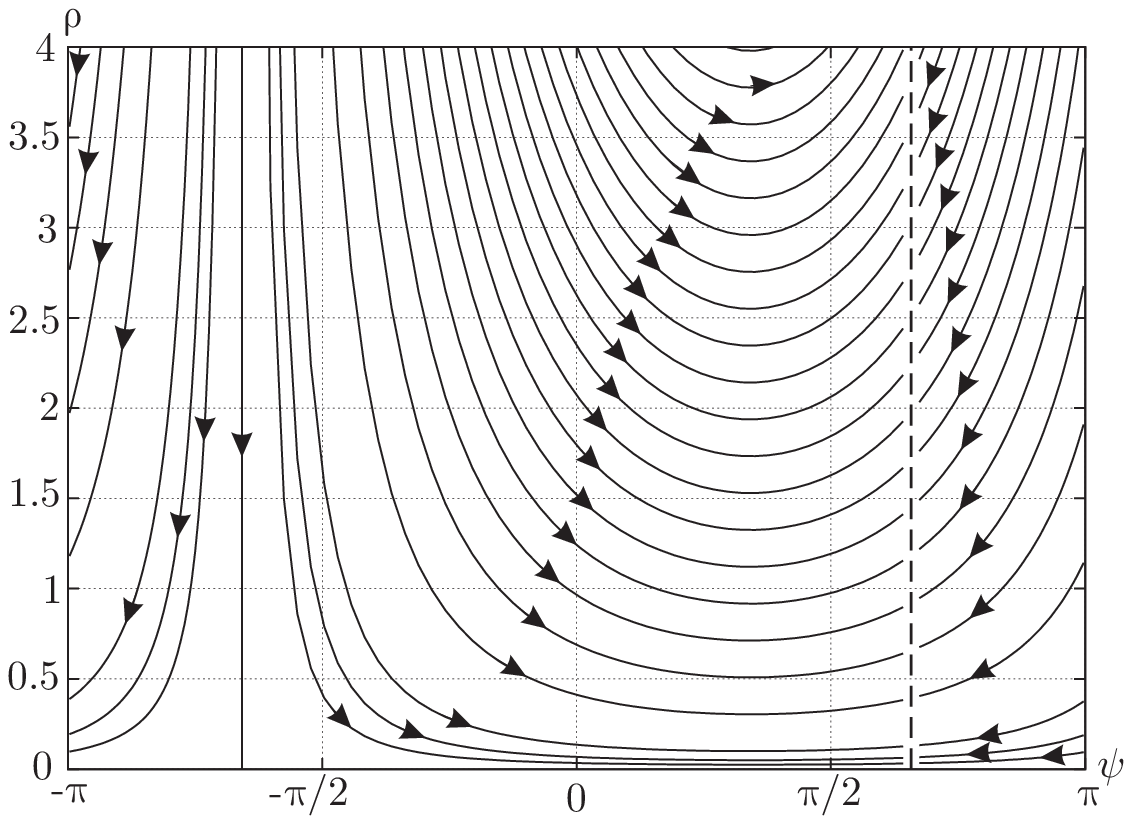}
			\includegraphics[width=0.32\linewidth]{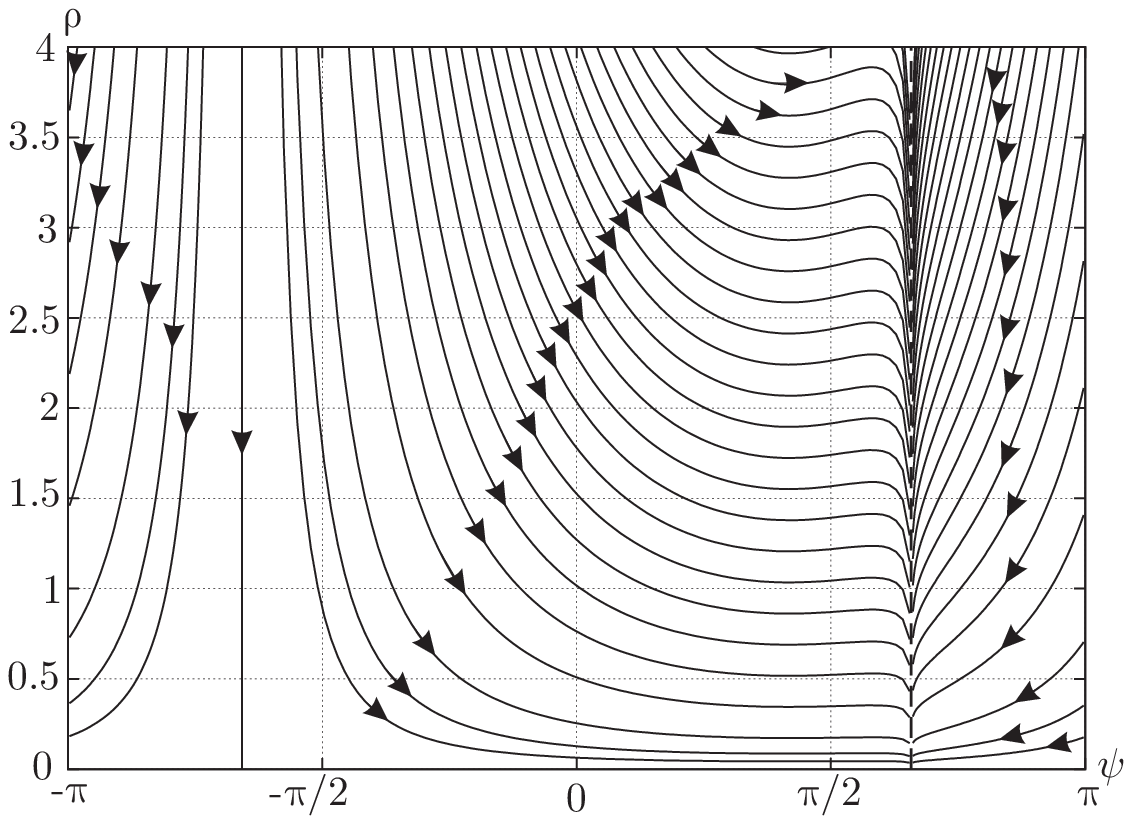}\\
			a) \hspace{4.5cm} b) \hspace{4.5cm} c)
			\caption{a) $\delta = -0.5$, b) $\delta = -1$, c) $\delta = -1.1$}\label{Phase_port13}
		\end{center}
		\vspace{-2mm}
	\end{figure}

	\begin{Note}
		By virtue of the symmetry \eqref{symmetry}, the phase portrait for some $\delta = -\delta _0 < 0$ can be obtained by a mirror reflection of the phase portrait for $\delta = \delta _0$ relative to $\psi = 0$ and by changing the direction of motion along the trajectories.
	\end{Note}
	
	2.2.1. If $\vert \delta \vert < 1$, the phase variable $\rho$ infinitely increases near $\psi _\pm$. The asymptotes $\psi = \psi _\pm$ are always separated from each other regardless of the relationship between the parameters $r$, $\sigma$, and $\varphi _0$. Hence, the qualitative behavior of the system trajectories on the plane $(\rho,\, \psi)$  (see Fig. \ref{Phase_port13}a) is also independent of the relationships between these parameters and is the same as the behavior considered in the case $\sin \varphi _0 = 0$, $r > \sigma$ (see Fig. \ref{Phase_port11}). Thus, for $\sin \varphi _0 \neq 0$ and $r < \sigma \vert \cos \varphi _0 \vert$ a partial stabilization can be performed only in finite time.
	
	2.2.2. The condition $\vert \delta \vert = 1$ is equivalent to the equality $r = \sigma$. In this case, the right-hand sides of equations \eqref{phi_eq1} vanish simultaneously for $\psi = \pi + \varphi _0$. When $\delta = -1$, the asymptote $\psi = \psi _{-}$ disappears, and its place is taken by a family of fixed points; there are no qualitative changes in a neighborhood of the asymptote $\psi = \psi _{+}$ (see Fig. \ref{Phase_port13}b). Similarly, when $\delta = 1$, the lines $\psi = \psi _{+}$ correspond to a family of fixed points and $\psi = \psi _{-}$ is an asymptote. To perform a stability analysis of these fixed points, we represent the first equation of \eqref{phi_eq1} as
	\begin{gather}
	m \rho \dot{\psi} = -\frac{r}{1 + \tan ^2 \frac{\psi}{2}} \tau _{+} (\psi)\tau _{-} (\psi) \label{psi_tau}.
	\end{gather}

	Let us analyze the stability of the family of fixed points $\psi = \psi _{-}$ for $\delta=-1$. For this purpose, we linearize equation \eqref{psi_tau} in a neighborhood of $\psi = \psi _{-}$
	\begin{gather}
	m \rho \dot{\Delta \psi} = - \frac{r}{1 + \tan^2 \frac{\psi _{-}}{2}} \tau _{+}(\psi _{-}) \frac{\sqrt{1 - \varkappa}}{2 \cos^2 \frac{\psi _{-}}{2}} \Delta  \psi, \quad \psi = \psi _{-} + \Delta  \psi .
	\end{gather}
	Since the coefficient of $\Delta \psi$ is negative, the fixed points of the family $\psi = \psi_{-}$ are stable. In a similar way, it can be shown that the fixed points of the family $\psi = \psi _{+}$ are unstable for $\delta = 1$. Since $\sign \delta = \sign \tan \varphi _0$, the system has the above family of stable fixed points for $\tan \varphi _0 <0$, and the family of unstable fixed points for $\tan \varphi _0 > 0$. Thus, a partial stabilization is possible in infinite time when $\sin \varphi _0 \neq 0$, $r > \sigma \vert \cos \varphi _0 \vert$ and $\delta = -1$.
	
	2.2.3. Consider the behavior of the system for $\vert \delta \vert > 1$. The phase portrait corresponding to $\delta < - 1$ is shown in Fig. \ref{Phase_port13}c. The behavior in a neighborhood of the straight line $\psi = \psi _{+}$ does not change qualitatively. In contrast to the cases considered above, the line $\psi = \psi_{-}$ becomes a discontinuity of the integral~$\ol{G}$. Note that due to equation \eqref{psi_tau} $\dot{\psi} > 0$ for $\psi \in (\psi_{+},\, \psi _{-})$ and $\dot{\psi} < 0$ for $\psi \in [-\pi / 2,\, \psi_{+}) \cup (\psi _{-},\, \pi / 2]$. Thus, all trajectories of \eqref{phi_eq1} tend to the point $\psi = \psi _{-}$, $\rho = 0$ on a given level set of the integral $\ol{G}$.
	
	The point $\psi = \psi _{-}$, $\rho = 0$ is the singular point of \eqref{phi_eq1}. This singularity may be due to either the choice of polar coordinates or the existence of an essential singular point in the system. In order to define the type of singularity, it is necessary to examine the value of the limit $\lim \limits _{\rho \rightarrow 0} \dot{\rho}$ depending on $\psi$. This analysis for the system considered shows that the point $\psi = \psi _{-}$, $\rho = 0$ is an essential singular point and all trajectories of the system converge to it.
	
	Consider the attainability of the point $\psi = \psi _{-}$, $\rho = 0$ in finite/infinite time. Let us eliminate $\rho$ from the first equation of \eqref{phi_eq1} using the integral \eqref{int_rhoPsi}
	\begin{gather}
	\dot{\psi} = \frac{1}{m G} \frac{1}{(1+\tan^2 \frac{\psi}{2})^2} \vert \tau_{+} (\psi) \vert ^{2-\delta} \vert \tau_{-} (\psi) \vert ^{2+\delta} \label{Psi13}.
	\end{gather}
	Equation \eqref{Psi13} can be approximated by
	\begin{gather}
	\dot{\psi} = \Lambda (\psi - \psi_{-})^{2 + \delta} \label{approx_Psi}
	\end{gather}
	using the Taylor series expansion of the function $\tau_{-}(\psi)$ in a neighborhood of $\psi = \psi _{-}$. Without loss of generality we set $\psi - \psi_{-} > 0$. Since $2 + \delta < 1$, the solution of \eqref{approx_Psi} is the following power function:
	\begin{gather}
	\psi = \psi_{-} + ( -(1 + \delta) (\Lambda t  + C))^{-1 / (\delta + 1)},
	\end{gather}
	whence it is clear that the line $\psi = \psi _{-}$ is attained in finite time.
	
	Let us consider the behavior of the angular velocity $\Omega$ as the line $\psi = \psi _{-}$ is approached. Expression \eqref{phi_eq3} can be written as
	\begin{gather}
	\Omega = \frac{1}{I_r} \left( C_{xy} - (m \rho^2 + \ol{b}) \dot{\psi} \right). \label{last_eq}
	\end{gather}
	It is seen from \eqref{approx_Psi} and \eqref{last_eq} that for $\delta \in (-2,\, -1)$ the derivative $\dot{\psi}$ tends to zero as $\psi = \psi _{-}$ is approached, hence, $\Omega \rightarrow \dfrac{C_{xy}}{I_r}$. If $\delta = -2$, then $\dot{\psi} = \Lambda$, hence, $\Omega \rightarrow \dfrac{C_{xy} - \ol{b} \Lambda}{I_r}$. If $\delta < - 2$, then $\dot{\psi} \rightarrow \infty$, hence, $\Omega \rightarrow \infty$. Thus, a partial stabilization is possible in finite time for $\sin \varphi _0 \neq 0$, $r > \sigma \vert \cos \varphi _0 \vert$ and $\vert \delta \vert > 1$.
\end{proof}

\end{document}